\documentclass[11pt,reqno]{amsart}
\usepackage{a4wide}
\usepackage[english]{babel}
\usepackage[utf8]{inputenc}
\usepackage{graphicx}
\usepackage{amsmath,amssymb,amsfonts,amsthm,mathrsfs,enumerate,cite,xcolor}
\usepackage[pdfpagelabels,linktocpage]{hyperref}
\usepackage{url}

\theoremstyle{plain}
\newtheorem{theorem}{Theorem}[section]
\newtheorem{lemma}[theorem]{Lemma}
\newtheorem{prop}[theorem]{Proposition}

\newtheorem{remark}[theorem]{Remark}

\theoremstyle{definition}
\theoremstyle{remark}
\numberwithin{equation}{section}

\newcommand{\RR}{\mathbb{R}}
\newcommand{\NN}{\mathbb{N}}
\newcommand{\II}{I\hspace{-0.2em}I}
\newcommand{\llc}{\,\text{\Large{\reflectbox{$\lrcorner$}}}}

\renewcommand{\d}{\partial}
\renewcommand{\div}{\,\mathrm{div}}
\newcommand{\dint}{\,\mathrm{d}}
\newcommand{\veps}{\varepsilon}

\newcommand{\nn}{\nonumber}
\newcommand{\wto}{\rightharpoonup}
\newcommand{\lara}[1]{\langle #1 \rangle}
\newcommand{\sbs}{\subseteq}
\newcommand{\id}{\mathrm{id}}
\newcommand{\loc}{\mathrm{loc}}
\newcommand{\supp}{\mathrm{supp}}

\newcommand{\cA}{\mathcal{A}}

\newcommand{\dist}{\mathrm{dist}}
\newcommand{\diam}{\mathrm{diam}}
\newcommand{\Radon}{\mathcal{M}}
\newcommand{\Haus}{\mathcal{H}}
\newcommand{\MM}{\mathbb{M}}
\newcommand{\curr}[1]{[\![ #1 ]\!]}
\newcommand{\encvol}{\,\mathrm{encvol}}

\newcommand{\one}{\mathrm{1}} 
\newcommand{\two}{\mathrm{2}} 
\newcommand{\CH}{\textup{CH}}
\newcommand{\E}{\mathcal{E}}
\newcommand{\F}{\mathcal{F}}
\newcommand{\FCH}{\F_\CH}
\newcommand{\FCHN}{\F_{\CH,N}}
\newcommand{\ECH}{\E_\CH} 
\newcommand{\Will}{\mathcal{W}}

    \let\TeXchi\chi
\newbox\chibox
\setbox0 \hbox{\mathsurround0pt $\TeXchi$}
\setbox\chibox \hbox{\raise\dp0 \box 0 }
\def\chi{\copy\chibox}

\title[Existence of varifold minimizers]{Existence of varifold minimizers for the \\ multiphase Canham-Helfrich functional}

\author[K.\,Brazda]
{Katharina Brazda}
\address[K.\,Brazda]{Faculty of Mathematics, University of Vienna, Oskar-Morgenstern-Platz 1,1090 Wien, Austria}
\email{katharina.brazda@univie.ac.at}

\author[L.\,Lussardi]
{Luca Lussardi}
\address[L.\,Lussardi]{Dipartimento di Scienze Matematiche ``G.L.\,Lagrange'', Politecnico di Torino, c.so Duca degli Abruzzi 24, I-10129 Torino, Italy}
\email[]{luca.lussardi@polito.it}

\author[U.\,Stefanelli]
{Ulisse Stefanelli}
\address[U.\,Stefanelli]{Faculty of Mathematics, University of Vienna, Oskar-Morgenstern-Platz 1,1090 Wien, Austria, Vienna Research Platform on Accelerating
  Photoreaction Discovery, University of Vienna, W\"ahringerstra\ss e 17, 1090 Wien, Austria, and Istituto di Matematica Applicata e Tecnologie Informatiche ``E.\,Magenes'', via Ferrata 1, 27100 Pavia, Italy.}
\email[]{ulisse.stefanelli@univie.ac.at}

\begin{document}

\baselineskip3.4ex

\vspace{0.5cm}
\begin{abstract}
We address the minimization of the Canham-Helfrich functional in presence of multiple phases. The problem is inspired by the modelization of heterogeneous biological membranes, which may feature variable bending rigidities and spontaneous curvatures. With respect to previous contributions, no symmetry of the minimizers is here assumed. Correspondingly, the problem is reformulated and solved in the weaker frame of oriented curvature varifolds. We present a lower semicontinuity result and prove existence of single- and multiphase minimizers under area and enclosed-volume constrains. Additionally, we discuss regularity of minimizers and establish lower and upper diameter bounds. 
{\small 

}
\end{abstract}

\maketitle

\section{Introduction} 

The equilibrium shape of a biological membrane is described by the classical {\sc Canham} \cite{Canham:70} \& {\sc Helfrich} \cite{Helfrich:73} theory as the result of the constrained minimization of its bending energy. Given a smooth closed surface $M$ in $\RR^3$, its {\it Canham-Helfrich energy} reads 
\begin{equation}\label{CH}
\ECH(M)=\int_M \left(\frac{\beta}{2}\,(H-H_0)^2+\gamma\, K\right)\dint\Haus^2
\end{equation}
where $\beta,\gamma$, and $H_0\in \RR$ are given material constants, $H$ is the mean curvature of $M$, and $K$ is the Gauss  curvature of $M$. The minimization is constrained to surfaces of fixed area and fixed enclosed volume. The preservation of area reflects the inextensibility of the membrane material and the volume constraint arises from the osmotic pressure balance between the vesicle and its surroundings \cite{Boal}. In contrast to the classical {\it Willmore energy} \cite{KuwertSchaetzle:12,Riviere:08,Schygulla:12,Simon:93} 
\begin{equation}\label{Willmore}
\Will(M)=\frac{1}{4}\int_MH^2\dint\Haus^2,
\end{equation}
the Canham-Helfrich energy \eqref{CH} for $H_0\neq 0$ implicitly depends on the orientation of $M$ as well, as its choice determines the sign of $H$. 

Biomembranes consist of single or double layers of lipids and frequently also contain other molecules. For example, cell membranes are lipid bilayers with embedded cholesterols and proteins. The molecule concentration locally alters the mechanical properties of the membrane by changing the bending rigidities $\beta,\gamma$ and the spontaneous curvature  $H_0$ \cite{Helfrich:73}.  Thereby, the composition of the membrane affects its shape.  Moreover, since lipid membranes behave as two-dimensional fluids \cite{SingerNicolson:72}, the constituents are not fixed within the membrane but rather can migrate due to surface diffusion and electrochemical dynamics. The result is an interplay between molecule concentrations, influencing the local membrane geometry, and geometry, influencing concentrations \cite{Boal, McMahon:05}.

In specific situations, biomembranes with well-separated domains of characteristic composition are observed \cite{Baumgart:03}. In the case of two of such {\it phases}, one can model the membrane as the union of compact surfaces  $M^{\one}$ and $M^{\two}$ in $\RR^3$ that overlap precisely at their boundary. The associated energy is the sum of the bending energies \eqref{CH} of the individual phases, augmented by an additional {\it line tension} contribution penalizing phase interfaces \cite{JuelicherLipowski:96}. The membrane morphology is hence described by the minimization of the {\it sharp-interface, two-phase Canham-Helfrich energy} 
\begin{equation}\label{CH2}
\E_{\CH,2}(M^{\one}, M^{\two})=\sum_{i=\one,\two}\int_{M^i}
\left(\frac{\beta^i}{2}\,(H-H_0^i)^2+\gamma^i\, K\right)\dint\Haus^2 
+\sigma\int_{\Gamma^{\one,\two}}\dint \Haus^1.
\end{equation}
Here, $\beta^i$, $\gamma^i$, and $H_0^i$ are the material parameters of the individual phases $M^i$ ($i=\one,\two$), $\sigma>0$ is the line-tension coefficient, and the phase boundary $\Gamma^{\one,\two}= M^{\one}\cap M^{\two}$ is assumed to be a union of closed curves. By assuming $\d M=\emptyset$ one has that the two-phase membrane $M=M^{\one}\cup M^{\two}$ is closed. 

In addition to the constraint on the volume enclosed by $M$, one fixes the areas of the individual phases $M^{\one}$ and $M^{\two}$. This two-phase model can be extended to more phases by letting $M=\bigcup_{i=1}^NM^i$, adding all individual Canham-Helfrich contributions, and penalizing all interfaces.

The focus of this paper is on investigating existence of equilibrium solutions associated to the multiphase Canham-Helfrich functional \eqref{CH2} under enclosed-volume and phase-area constraints. This variational problem has already been addressed by {\sc Choksi, Morandotti, \& Veneroni} \cite{ChoksiVeneroni:13,ChoksiMorandottiVeneroni:13}, who nevertheless  focused on  the case of {\it axisymmetric} membranes and phases. This assumption allowed them to reformulate the problem in terms of generating curves in the plane, a fact that is crucially exploited in the analysis. Still, by imposing rotational symmetry, many biologically relevant shapes  may be  ruled out \cite{Baumgart:03}. 

Our goal is hence to  remove this restriction and to study multiphase minimizers without a-priori  assumptions on their  symmetry. The price to pay for such generality is that the strong formulation \eqref{CH2} has to be relaxed. In particular, we resort in rephrasing the problem in the general framework of {\it curvature varifolds} \cite{Hutchinson:86, Mantegazza:96}, which are a geometric measure theory generalization of surfaces \cite{Allard:72,Simon:83}. This asks for defining a corresponding {\it generalized Canham-Helfrich functional}, see \eqref{eq:FCH} below. Such generalization is not  entirely new. In  fact, the minimization of geometric functionals involving curvature was {\sc Hutchinson}'s \cite{Hutchinson:86} motivation to introduce curvature varifolds in the first place. Notice that the Willmore functional \eqref{Willmore} is  readily  defined for integral varifolds with locally bounded first variation in the sense of {\sc Allard} \cite{Allard:72}, because its integrand can be expressed solely in terms of the mean curvature vector. However, the density of the Canham-Helfrich functional \eqref{CH} with $\gamma\neq 0$ contains more components of the second fundamental form and, with $H_0\neq 0$, also depends on the orientation. For these reasons and in view of the presence of interfaces in the multiphase energy \eqref{CH2}, we combine here Hutchinson's \cite{Hutchinson:86} oriented varifolds and  {\sc Mantegazza}'s \cite{Mantegazza:96} curvature varifolds with boundary, formulating the variational problem in the setting of {\it oriented curvature varifolds} {\it with boundary}, see \eqref{eq:AVo}.

Our main result is the existence of minimizers for the varifold reformulation \eqref{eq:FCH} of the Canham-Helfrich functional, both in the single- and in the multiphase regime. We proceed along the blueprint of the classical Direct Method. Lower semicontinuity with respect to curvature varifold convergence follows from convexity, upon checking that the energy bounds the $L^2$-norm of the second fundamental form. Such convexity holds under the qualification on the material parameters $-6\beta<5\gamma<0$, see \eqref{eq:bounds}, which is more restrictive than the one from \cite{ChoksiVeneroni:13,ChoksiMorandottiVeneroni:13}, but is still in the range of experimental data, see \cite[Formula (1.4)]{BellettiniMugnai:10} and the references therein. The existence of a single-phase minimizer with fixed area and  fixed  enclosed volume is proved in Theorem \ref{thm:single}. The multiphase case is tackled by additionally imposing phases to have  individually  fixed areas and not to overlap. Properly reformulated in the varifold setting, these conditions still allow for proving an existence result, see Theorem \ref{thm:multi}.  Eventually,  we show that both single- and multiphase minimizers are $C^2$-rectifiable (Lemma \ref{lem:menne}) and have bounded Willmore energy (Lemma \ref{lem:Willmore}). This last fact allows us to prove that the diameter of minimizers is bounded from above as well as away from zero (Lemmas \ref{lem:lowerdiam}  and  \ref{lem:lowerdiam2}). In the case of multiple phases, the lower bounds on the diameter apply to each individual phase. 

 We would like to  put our contribution in perspective by briefly reviewing the available literature on the minimization of Canham-Helfrich functionals, that is, \eqref{CH} or \eqref{CH2} for $H_0 \not = 0$. As already mentioned, the axisymmetric case has been discussed in \cite{ChoksiVeneroni:13} for single phases and in \cite{ChoksiMorandottiVeneroni:13} in the case of multiple phases with sharp interfaces. Also under rotational-symmetry assumptions, {\sc Helmers} \cite{Helmers:13, Helmers:15} adopts a surface phase-field approach and presents $\Gamma$-convergence results to the sharp-interface limit for two-phase membranes with possibly discontinuous tangent vectors. To the best of our knowledge, \cite{ChoksiMorandottiVeneroni:13,Helmers:13,Helmers:15} currently are the only analysis results for multiphase Canham-Helfrich variational problems. Single-phase minimizers without symmetry assumptions have been recently addressed by {\sc Eichmann} \cite{Eichmann:19,Eichmann:19b} and {\sc Mondino \& Scharrer} \cite{MondinoScharrer:19}. In \cite{Eichmann:19} specific Dirichlet boundary conditions are considered, by which the problem can be formulated in terms of branched immersions. The analysis in \cite{Eichmann:19b} focuses on closed immersions with prescribed genus $g$ whereas \cite{MondinoScharrer:19} deals with the case of weak and possibly branched immersions of a 2-sphere in $\RR^3$ (which has $g=0$). The minimization under additional confinement and connectedness constraints is discussed by {\sc M\"{u}ller \& R\"{o}ger} \cite{MuellerRoeger:14} and {\sc Dondl \& Wojtowytsch} \cite{DondlWojtowytsch:17, Wojtowytsch:17}, though with emphasis on the Willmore energy \eqref{Willmore}. In this context we also mention the recent contribution by {\sc Novaga \& Pozzetta} \cite{NovagaPozzetta:19}, where Mantegazza's curvature varifolds with boundary are employed to study connected Willmore minimizers for Dirichlet data. Based on currents instead of varifolds, {\sc Delladio} \cite{Delladio:97} introduced special generalized Gauss graphs as an alternative theoretical framework for the minimization of curvature-dependent functionals. Concerning approximations for the single-phase Canham-Helfrich variational problem \eqref{CH}, a bulk diffuse-interface formulation is addressed in {\sc Bellettini \& Mugnai} \cite{BellettiniMugnai:10}. Moreover, the Canham-Helfrich energy is obtained as limit of micro- and mesoscopic energies by {\sc Peletier \& R\"oger} \cite{PeletierRoeger:14}, see also \cite{LussardiRoeger:16,LussardiPeletierRoeger:14,Lussardi:19}. The multiphase Canham-Helfrich variational problem has also been studied numerically; for example, \textsc{Barrett, Garcke, \& N\"urnberg} \cite{BarrettGarckeNuernberg:17, BarrettGarckeNuernberg:18} present a finite element approach to the gradient flow formulation of the evolution problem associated to the two-phase energy \eqref{CH2}.

Before closing this introduction, let us mention that we do not assume here that the genus of the membrane is prescribed. Indeed, allowing for equilibrium shapes with different topology is beneficial from a modeling perspective. Due to the phase-interface penalization in the energy \eqref{CH2}, a union of several disconnected single-phase membranes might have the lowest energy in certain parameter regimes. This situation is natural as it corresponds to vesicle budding or cell division. On the contrary, by prescribing $g$ for closed single-phase membranes $M$, one would drop the Gauss curvature term in the Canham-Helfrich functional \eqref{CH}, because its integral would be a topological invariant by the Gauss-Bonnet theorem: $\int_MK\dint\Haus^2=4\pi(1-g)$. Yet, when dealing with heterogeneous membranes, the bending modulus $\gamma$ will generally not be constant and its contribution plays an essential role in our model.

The paper is organized as follows. Some preliminary material including notation is gathered in Section \ref{sec:notation}. The varifold formulation of the Canham-Helfrich problem is detailed in Section \ref{sec:formulation}, where we also discuss lower semicontinuity  (Theorem \ref{thm:lsc}).  The existence of single-phase (Theorem \ref{thm:single}) and multiphase minimizers (Theorem \ref{thm:multi}) is eventually proved in Section \ref{sec:existence}, which we conclude with some comments on regularity and diameter bounds for minimizers.

\section{Notation and preliminaries}\label{sec:notation}

\subsection{Notation for measures and norms} 
 
Let $X$ be a locally compact and separable metric space, which in the applications  of  this paper can be identified with a nonempty Euclidean subset $X\sbs\RR^\alpha$ for some $\alpha\in\NN$. The class of nonnegative (finite) Radon measures on $X$ is denoted by $\mathcal M(X)$ and we write $\mathcal M^\beta(X)$ with $\beta\in\NN$ for its $\RR^\beta$-valued counterpart. We write $\|\mu\|\in\Radon(X)$ for the total variation measure of $\mu\in\Radon^\beta(X)$. Depending  on the context, $|\cdot|$ stands for the absolute value, the Euclidean norm, the comass norm (see Section \ref{sec:currents}), or the Frobenius norm.  For  example, if $A\in\RR^{n\times n\times n}$ with $n\in\NN$, then $|A|^2=\sum_{i,j,k=1}^nA_{ijk}^2$. 
By $B_r(x):=\{y\in\RR^n:\:|x-y|<r\}$ we denote the open ball of radius $r>0$ and centered at $x\in\RR^n$. If $M\sbs\RR^n$ then its diameter is given by $\diam(M):=\sup_{x,y\in M}|x-y|$ and $\dist(x,M):=\inf_{y\in M}|x-y|$ defines the distance of $x\in\RR^n$ to $M$. The characteristic function of a set $U\sbs X$  denoted $\chi_U\colon X\to\{0,1\}$ is defined via $\chi_U(u)=1$ if $u\in U$ and $\chi_U(u)=0$ otherwise. 
If $\mu\in\Radon(X)$ then $L^p_\mu(X;\RR^\beta)$ with $p\in[1,\infty)$ and $\beta\in\NN$ is the space of equivalence classes of $\mu$-measurable functions $f\colon X\to\RR^\beta$ with finite norm $\|f\|_{L^p_\mu(X)}:=\left(\int_X |f(x)|^p\dint\mu(x)\right)^{\frac{1}{p}}$ and we write $f\in L^p_{{\rm loc},\mu}(X;\RR^\beta)$ whenever $\|f\|_{L^p_\mu(K)}<\infty$ for all compact sets $K\subset\subset X$.

If $\mu_h,\mu\in\mathcal M(X)$ and we let $h\to\infty$, then we say $\mu_h\wto^\ast\mu$ in $\mathcal M(X)$  whenever 
$\mu_h(\varphi)=\int_X\varphi\dint\mu_h\to\int_X\varphi\dint\mu=\mu(\varphi)$ for all continuous functions with compact support, $\varphi\in C_c(X)$. The convergence of vector-valued Radon measures $\mu_h\wto^\ast\mu$ in $\Radon^\beta(X)$ is defined analogously upon testing on $\varphi\in C_c(X;\RR^\beta)$ with duality $\mu(\varphi)=\int_X\langle\varphi,\dint\mu\rangle$. Here $\langle\cdot,\cdot\rangle$ stands for the Euclidean inner product of vectors, but we use the same bracket to indicate the duality in the context of multivectors (again Section \ref{sec:currents}). For $\mu\in\Radon(X)$ and a $\mu$-measurable map $f\colon X\to Y$ with another locally compact and separable metric space $Y$, the pushforward measure of $\mu$ by $f$ is defined by $(f_\sharp\mu)(E):=\mu(f^{-1}(E))$ for all Borel sets $E\sbs Y$. If $f$ is continuous and proper, meaning $f^{-1}(K)$ is compact for all compact $K\sbs Y$, then $f_\sharp\mu\in\Radon(Y)$. 
 
Further notation will be explained in the remainder of this section and in the main text.

\subsection{Convergence as measure-function pairs}\label{sec:mfp}

In this paragraph we review the notion of measure-function  pair  convergence and the corresponding lower semicontinuity of integral functionals. For details we refer to Hutchinson \cite{Hutchinson:86}.

A {\it measure-function pair} $(\mu,B)$ over $X\subseteq\RR^\alpha$ with values in $\RR^\beta$ ($\alpha,\beta\in\NN$) consists of a measure $\mu\in\Radon(X)$ and a function $B\in L^1_{\loc,\mu}(X;\RR^\beta)$.  
Convergence 
$$
(\mu_{h},B_{h})\wto^\ast (\mu,B)
$$ 
as measure-function pairs is defined by the two conditions 
$$
\mu_{h}\wto^\ast\mu\quad\text{in}\quad\Radon(X)\qquad\text{and}\qquad B_h\mu_{h}\wto^\ast B\mu\quad\text{in}\quad\Radon^\beta(X).
$$

\begin{theorem}[Integral functionals of measure-function pairs]\label{thm:mfp} 
Let $X\subseteq\RR^\alpha$ be nonempty and $f\colon X\times\RR^\beta\to\RR$  satisfy  
\begin{itemize}
\item[\rm(a)] $f$ is continuous,
\item[\rm(b)] $f(y,\cdot)$ is convex for any $y\in X$, and
\item[\rm(c)] $f(y,B) \ge c|B|^2-d$ holds for any $(y,B)\in X\times\RR^\beta$ and for some $c,d>0$.
\end{itemize}
Let $(\mu_{h},B_{h})$ be a sequence of measure-function pairs over $X\subseteq\RR^\alpha$ with values in $\RR^\beta$ and let $J$ be the functional
$$
(\mu_h,B_h)\mapsto J(\mu_h,B_h):=\int_X f(y,B_h(y))\dint\mu_h(y).
$$
Then, with $h\to\infty$, the following assertions hold:
\begin{itemize}
\item[\rm(i)] {\em (Compactness)} If there exists $\mu\in\Radon(X)$ such that $\mu_{h}\wto^\ast\mu$ in $\Radon(X)$ and if $J(\mu_{h},B_{h})$ is bounded, then there exists $B\in L^1_{\loc,\mu}(X;\RR^\beta)$, such that up to a subsequence (not relabeled), $(\mu_h,B_h)\wto^\ast (\mu,B)$ as measure-function pairs;
\item[\rm(ii)] {\em (Lower semicontinuity)} If $(\mu_{h},B_{h})\wto^\ast (\mu,B)$ as measure-function pairs, then 
$$
J(\mu,B)\leq\liminf_{h\to\infty}J(\mu_{h},B_{h}).
$$
\end{itemize}
\end{theorem}

\subsection{Currents}\label{sec:currents}

We briefly recall some points of the theory of currents. For details see \cite{Simon:83}. 

Let $\mathcal D^m(\RR^n)$ be the set of all smooth $m$-forms with compact support in $\RR^n$, where $m\in\NN_0$ and $n\in\NN$ with $m\leq n$. The (topological) dual space of $\mathcal D^m(\RR^n)$ is the space of all {\it $m$-currents on $\RR^n$}, denoted by $\mathcal D_m(\RR^n)$. In particular, $0$-currents coincide with the distributions on $\RR^n$. 

In accordance with the dual nature of $\mathcal D_m(\RR^n)$, if $T_h,T\in \mathcal D_m(\RR^n)$  we say that $T_h \wto^\ast T$ if $T_h(\omega) \to T(\omega)$ for all $\omega \in \mathcal D^m(\RR^n)$. If $m\geq 1$ then the {\it boundary} $\partial T\in\mathcal D_{m-1}(\RR^n)$ of $T\in \mathcal D_m(\RR^n)$ is defined via Stokes' theorem  as  $\partial T(\omega):=T(\!\dint\omega)$, where $\dint \omega $ denotes the exterior differential of the $(m-1)$-form $\omega$. The {\it mass} of $T\in\mathcal D_m(\RR^n)$ is its dual norm
$$
\mathbb M_T:=\sup_{\omega\in \mathcal D^m(\RR^n),\:\|\omega\|_{L^\infty(\RR^n)}\leq 1}T(\omega),
$$
where 
$\|\omega\|_{L^\infty(\RR^n)}=\sup_{x\in\RR^n}|\omega(x)|$ with 
$|\omega(x)|=\sum_{1\leq i_1<\ldots<i_m\leq n}\langle\omega(x),e_{i_1}\wedge\ldots\wedge e_{i_m}\rangle$ for basis vectors $e_i$ of $\RR^n$, the comass norm of the $m$-covector $\omega(x)$. The mass $\mathbb M_T(\Omega)$ of $T\in\mathcal D_m(\Omega)$ with $\Omega\sbs\RR^n$ open is defined analogously upon restricting to test forms with $\supp(\omega)\subset\subset\Omega$. 

In order to define rectifiable currents we recall the notion of a rectifiable set. Let $M$ be a Borel subset of $\RR^n$ and denote the $m$-dimensional Hausdorff measure in $\RR^n$ by $\Haus^m$. We say that $M$ is {\it (countably) $m$-rectifiable}, if  $M=M_0 \cup \bigcup_{k=1}^\infty M_k$ where $\mathcal H^m(M_0)=0$ and $M_k$ is contained in the image of a Lipschitz function $f_k \colon \RR^m\to \RR^n$ for all $k\in\NN$. It turns out that for $\mathcal H^m$-a.e.\,$x\in M$ there exists an {\it approximate tangent space}, which we denote by $T_xM$ as in the smooth setting.  

We say that $T\in\mathcal D_m(\RR^n)$ is an  {\it $m$-rectifiable current with integer multiplicity} if there exists
\begin{itemize}
	\item[\rm(a)] an $m$-rectifiable set $M$ in $\RR^n$,
	\item[\rm(b)] an {\it orientation} of $M$, that is, a Borel map $\xi$ defined on $M$ that to $\mathcal H^m$-a.e.\,$x\in M$ assigns a unit simple $m$-vector $\xi(x)$ which spans $T_xM$, and 
	\item[\rm(c)] an integer-valued {\it multiplicity}, that is, a locally $(\Haus^m\llc M)$-summable function $j\colon M\to \mathbb Z$, where $\Haus^m\llc M$ denotes the $\mathcal H^m$ measure restricted to $M$,  
\end{itemize}
such that 
$$
T(\omega)=\int_M\langle\omega(x),\xi(x)\rangle j(x)\,\dint\mathcal H^m(x) \quad\text{for all}\: \omega \in \mathcal D^m(\RR^n).
$$
In this case we write $T=[M,\xi,j]$. Finally we say that $T\in\mathcal D_m(\RR^n)$ is an {\it integral $m$-current} if both $T$ and $\partial T$ are rectifiable currents with integer multiplicity.

The theory of currents features the following compactness and lower semicontinuity results, see again \cite{Simon:83}:

\begin{theorem}[Federer-Fleming]\label{thm:FF} 
	Let $(T_h)$ be a sequence of integral $m$-currents on $\RR^n$ such that $\mathbb M_{T_h}+\mathbb M_{\partial T_h}$ is bounded and let $h\to\infty$. Then, up to a subsequence (not relabeled), $T_h\wto^\ast T$ in $\mathcal D_m(\RR^n)$ and $\partial T_h \wto^\ast  \partial T$ in $\mathcal D_{m-1}(\RR^n)$, where  the limit $T$ is  an  integral $m$-current. Moreover we have  $\mathbb M_T\le \liminf_{h\to\infty} \mathbb M_{T_h}$ and $\mathbb M_{\partial T}\le \liminf_{h\to\infty} \mathbb M_{\partial T_h}$. 
\end{theorem} 

The main insight of the Federer-Fleming  Theorem \ref{thm:FF} is not sequential compactness, namely  existence of convergent subsequences for integral currents with bounded mass, but the  fact  that the limit currents are again integral  (closedness). This remark applies to the forthcoming Theorems \ref{thm:Mantegazza_Thm6} and \ref{thm:Hutchinson_Thm3} as well.

\subsection{Varifolds}\label{sec:varifolds}

This paragraph is dedicated to a review of the theory of curvature varifolds, which are the weak version of surfaces that we need. For details we refer to \cite{Allard:72, Simon:83, Hutchinson:86}. 

Let $\Omega\sbs\RR^n$ be open and nonempty. Let $m,n\in\NN$ with $1\leq m\leq n$. An {\it $m$-varifold} in $\Omega$ is a Radon measure
$$
V\in \Radon(\Omega\times G_{m,n}).
$$
Here $G_{m,n}$ denotes the {\it Grassmannian}, which consists of all $m$-dimensional linear subspaces of $\RR^n$. An element $P\in G_{m,n}$ is identified with its associated orthogonal projection matrix $P\in\RR^{n\times n}$. This allows us to view $\Omega\times G_{m,n}$ as a subset of some  Euclidean space.

The pushforward measure of $V$ under the canonical projection map $\pi \colon \Omega \times G_{m,n} \to \Omega$, $\pi(x,P)=x$ defines the {\it mass measure} (or {\it weight}) of the varifold,
$$
\mu_V:=\pi_\sharp V\in\Radon(\Omega).
$$  
The {\it mass} of $V$ then reads 
$$
\mu_V(\Omega)=\int_{\Omega}\dint \mu_V=\int_{\Omega\times G_{m,n}}\dint V=V(\Omega\times G_{m,n}).
$$    
As Radon measures, varifolds are determined through their action on continuous functions with compact support in $\Omega\times G_{m,n}$.  Based on this duality one can introduce the following useful subclass: An {\it integral $m$-varifold} in $\Omega$ is  given by 
$$
V(\varphi)=\int_M\varphi(x,T_xM)\theta(x)\dint\Haus^m(x) \quad\text{for all}\:\varphi\in C_c(\Omega\times G_{m,n}),
$$
where $M\sbs\Omega$ is (countably) $m$-rectifiable and $\theta\colon M\to\NN$, called {\it multiplicity}, is locally summable w.r.t.\ $\Haus^m\llc M$. In this case we write $V=v(M,\theta)$ and we denote the space of all integral $m$-varifolds in $\Omega$ by $IV_m(\Omega)$. Its elements can be seen as the geometric measure theory analogue to the $m$-dimensional differentiable submanifolds in $\Omega$. 

By a clever application of the divergence theorem for manifolds, Hutchinson \cite{Hutchinson:86}, and later Mantegazza \cite{Mantegazza:96}, generalized the concept of the second fundamental form to the varifold setting. For our convenience we will follow Mantegazza's approach since it also gives a good definition of the boundary, see \cite{Mantegazza:96} for details. We say that $V\in IV_m(\Omega)$ is a {\it curvature varifold with boundary} and write $V\in AV_m(\Omega)$, if there exists 
$$
A^V\in L^1_{\loc,V}(\Omega\times G_{m,n};\RR^{n\times n\times n}) \quad\text{and}\quad \d V\in\Radon^n(\Omega\times G_{m,n})
$$
such that 
$$
\int_{\Omega\times G_{m,n}}\left(P_{ij}\d_j\varphi+(\d_{P_{jk}}\varphi)\,A_{ijk}^V+A_{jij}^V\,\varphi\right)(x,P)\dint V(x,P)=-\int_{\Omega\times G_{m,n}}\varphi(x,P)\dint (\d V)_i(x,P)
$$
for all $\varphi\in C_c^1(\Omega\times G_{m,n})$ and $1\leq i\leq n$. Here and in the following we employ summation convention, that is, we take the sum over all indices that occur twice. The curvature functions $A^V$ give rise to the {\it generalized second fundamental form} $\II^V\in L^1_{\loc,V}(\Omega\times G_{m,n};\RR^{n\times n\times n})$ via
$$
(\II^V(\cdot,P))_{ij}^k:=A_{ikl}^V(\cdot,P)P_{lj}
$$
and $\d V$ is the {\it boundary measure} of the varifold. The associated {\it generalized mean curvature vector} 
$\bar H^V\in L^1_{\loc,V}(\Omega\times G_{m,n};\RR^n)$
has the components
$$
\bar H_i^V:=A_{jij}^V.
$$
Allard's  {\it first variation}  \cite{Allard:72,Simon:83} of a varifold $V$ is the linear functional $\delta V\colon C_c^1(\Omega;\RR^n)\to \RR$,
$$
\delta V(X)=\int_{\Omega\times G_{m,n}}\div^PX(x)\dint V(x,P)
$$
with the tangential divergence 
$$
\div^P X:=P_{ij}\d_jX_i. 
$$
Insertion of $\varphi(x,P)=X_i(x)$ with $X\in C_c^1(\Omega;\RR^n)$ in Mantegazza's definition and summation over $1\leq i\leq n$ reveals that $V\in AV_m(\Omega)$ satisfies the {\it first variation formula} \cite{Mantegazza:96}
\begin{equation}\label{eq:Mantegazzaformula}
\delta V=\langle -\pi_\sharp\left(\bar{H}^V\,V+\d V\right),\cdot\rangle,
\end{equation}
that is,
$$
\delta V(X)=-\int_{\Omega\times G_{m,n}}\bar{H}^V_i(x,P)X_i(x)\dint V(x,P)-\int_{\Omega\times G_{m,n}}X_i(x)\dint (\d V)_i(x,P).
$$ 
Formula \eqref{eq:Mantegazzaformula} shows that a curvature varifold $V\in AV_m(\Omega)$ has locally bounded first variation. In particular, $\|\delta V\|\in\Radon(\Omega)$ with the total variation measure given by
$$
\|\delta V\|(\Omega'):=\sup_{X\in C_c^1(\Omega;\RR^n),\:\supp(X)\subset\subset\Omega',\:\|X\|_{L^\infty(\Omega)}\leq 1}\delta V(X)
$$
for open subsets $\Omega'\sbs\Omega$. Next we prove some preliminary estimates that we will need.

\begin{lemma}[First variation bounds] 
	Let $V\in AV_m(\Omega)$. Then we have 
	\begin{equation}\label{eq:HA}
	|\bar H^V|^2 \le 2|A^V|^2.
	\end{equation}
	Moreover, if in addition $A^V\in L^2_V(\Omega\times G_{m,n};\RR^{n\times n\times n})$, then 
	\begin{equation}\label{eq:firstvarbound}
	\|\delta V\|(\Omega)\leq  \sqrt{2\, \mu_V(\Omega)}\, \|A^V\|_{L^2_V(\Omega\times G_{m,n})}+\|\d V\|(\Omega\times G_{m,n}).
	\end{equation}
\end{lemma}

\begin{proof} 
	By Formula \eqref{eq:Mantegazzaformula} 
	\begin{align*}
	\|\delta V\|(\Omega)=\|-\pi_\sharp\left(\bar{H}^V\,V+\d V\right)\|(\Omega)
	&\leq\|\bar{H}^V\,V+\d V\|(\Omega\times G_{m,n})\\
	&\leq \|\bar{H}^V\,V\|(\Omega\times G_{m,n})+\|\d V\|(\Omega\times G_{m,n})
	\end{align*}
	where we notice that 
	\[
	\:\|\bar{H}^V\,V\|(\Omega\times G_{m,n})
	=\int_{\Omega\times G_{m,n}}|\bar{H}^V|\dint V= \|\bar{H}^V\|_{L^1_V{(\Omega\times G_{m,n})}}.
	\]
	The estimate \eqref{eq:HA} is easily obtained:
	$$
	{|\bar H^V|^2{=\sum_{i=1}^n\Big(\bar{H}^V_i\Big)^2}=\sum_{i=1}^n\Big(\sum_{j=1}^nA_{jij}^V\Big)^2
		\leq \sum_{i=1}^n\Big(2\sum_{j=1}^n(A_{jij}^V)^2\Big)
		\leq 2\sum_{i,j,k=1}^n(A_{ijk}^V)^2=2|A^V|^2}.
	$$
	Consequently, we have
	$$
	\|\bar{H}^V\|_{L^1_V{(\Omega\times G_{m,n})}}= \int_{\Omega\times G_{m,n}}|\bar{H}^V|\dint V
	\leq \sqrt{2} \int_{\Omega\times G_{m,n}}|A^V|\dint V=\sqrt{2}\,\|A^V\|_{L^1_V{(\Omega\times G_{m,n})}},
	$$
	showing that $A^V\in L^1_{V}(\Omega\times G_{m,n};\RR^{n\times n\times n})$ implies $\bar{H}^V\in L^1_{V} (\Omega\times G_{m,n};\RR^{n})$. If $A^V$ is square-integrable then H\"older's inequality gives
	\begin{align*}
	\|A^V\|_{L^1_V{(\Omega\times G_{m,n})}}
	&=\int_{\Omega\times G_{m,n}}|A^V|\dint V=\int_{\Omega}|A^V|\dint\mu_V\\
	&\leq\Big(\int_{\Omega}\dint\mu_V\Big)^{1/2}\Big(\int_{\Omega}|A^V|^{2}\dint\mu_V\Big)^{1/2}
	=(\mu_V(\Omega))^{1/2}\, \|A^V\|_{L^2_V{(\Omega\times G_{m,n})}},
	\end{align*}
	and thus 
	$$
    \|\bar{H}^V\|_{L^1_V{(\Omega\times G_{m,n})}}\leq \sqrt{2\,\mu_V(\Omega)}\, \|A^V\|_{L^2_V{(\Omega\times G_{m,n})}}.
    $$
   With $\|\delta V\|(\Omega)\leq \|\bar{H}^V\|_{L^1_V{(\Omega\times G_{m,n})}}+\|\d V\|(\Omega\times G_{m,n})$
   from above this yields \eqref{eq:firstvarbound}.
\end{proof}

One of the main results by Mantegazza \cite{Mantegazza:96} is the following compactness (closedness) theorem. 

\begin{theorem}[Compactness for curvature varifolds]\label{thm:Mantegazza_Thm6}
	Let $p>1$ and let $(V_h)$ be a sequence in $AV_m(\Omega)$. Assume that 
	$$
	\mu_{V_h}(\Omega)+\|A^{V_h}\|^p_{L_{V_h}^p(\Omega\times G_{m,n})}+\|\d V_h\|(\Omega\times G_{m,n})\le c
	$$
	for some $c>0$. Then there exists $V\in AV_m(\Omega)$, such that up to a subsequence (not relabeled), $V_h\wto^\ast V$ in $AV_m(\Omega)$. Moreover, $\d V_h\wto^\ast \d V$ in $\Radon^n(\Omega\times G_{m,n})$.
\end{theorem}

Here, {\it curvature varifold convergence}  
$$
V_h\wto^\ast V\quad\text{in}\quad AV_m(\Omega)
$$ 
is defined as the measure-function pairs convergence $(V_h,A^{V_h})\wto^\ast (V,A^{V})$ over $\Omega\times G_{m,n}$ with values in $\RR^{n\times n\times n}$, cf.\ Section \ref{sec:mfp}.

Varifolds including orientation were introduced by Hutchinson \cite{Hutchinson:86}. An {\it oriented $m$-varifold} in $\Omega$ is a Radon measure 
$$
V\in\Radon(\Omega\times G_{m,n}^{o}),
$$
with the {\it oriented Grassmannian} $G_{m,n}^{o}$ given by the set of all oriented $m$-dimensional linear subspaces of $\RR^n$. The elements $\xi\in G_{m,n}^{o}$ can be represented by $m$-vectors in $\RR^n$. The two-fold covering map 
$$
q\colon G_{m,n}^o\to G_{m,n},\quad q(\pm\xi)=P
$$ 
allows us to define the {\it unoriented counterpart} of the varifold $V$ as the pushforward measure of $V$ by  $\id_{\RR^n}\otimes q\colon \Omega\times G_{m,n}^o\to \Omega\times G_{m,n}$, which is a continuous and proper map. We shorten the notation and write
$$
q_\sharp V:=(\id_{\RR^n}\otimes q)_\sharp V\in\Radon(\Omega\times G_{m,n}).
$$
The masses of $V$ and $q_\sharp V$ coincide:
$$
\mu_{q_\sharp V}(\Omega)=\int_{\Omega\times G_{m,n}^o}\dint (q_\sharp V)=\int_{\Omega\times G_{m,n}}\dint V=\mu_V(\Omega).
$$ 

\begin{lemma}[Convergence of oriented and unoriented varifolds]\label{lem:orconv}
Let $h\to\infty$. If $V_h\wto^\ast V$ in $\Radon(\Omega\times G_{m,n}^o)$ then $q_\sharp V_{h}\wto^\ast q_\sharp V$ in $\Radon(\Omega\times G_{m,n})$. 
\end{lemma}
\begin{proof}
Let $\varphi\in C_c(\Omega\times G_{m,n})$ and consider $\psi\in C_c(\Omega\times G_{m,n}^o)$ given by 
$$
\psi(x,\xi):=\varphi(x,q(\xi))=\varphi(x,q(-\xi))=\psi(x,-\xi).
$$ 
Then, with $P=q(\pm\xi)$,
\begin{multline}
(q_\sharp V_{h})(\varphi)
	=\int_{\Omega\times G_{m,n}}\varphi(x,P)\dint(q_\sharp V_{h})(x,P)
	=\frac{1}{2}\int_{\Omega\times G_{m,n}^o}\varphi(x,q(\xi))\dint V_{h}(x,\xi)
	=\frac{1}{2}V_h(\psi) \\
	\to\frac{1}{2}V(\psi)
	=\frac{1}{2}\int_{\Omega\times G_{m,n}^o}\varphi(x,q(\xi))\dint V(x,\xi)
	=\int_{\Omega\times G_{m,n}}\varphi(x,P)\dint(q_\sharp V)(x,P)=(q_\sharp V)(\varphi),\nn
\end{multline}
where the factor $\frac{1}{2}$ comes from the symmetry of $\psi$.
\end{proof}

The classical divergence theorem for closed hypersurfaces of $\RR^n$ motivates us to assign an {\it enclosed volume} to an oriented $(n-1)$-varifold $V\in \Radon(\Omega\times G_{n-1,n}^o)$ by 
$$
\encvol(V):=\frac{1}{n}\int_{\Omega\times G_{n-1,n}^o}x\cdot(\ast\,\xi)\dint V(x,\xi).
$$
Here, $\ast$ stands for the Hodge operator from exterior calculus, which is a bijection between $m$- and $(n-m)$-(co-)vectors in $\RR^n$, $0\leq m\leq n$. In particular, in the relevant case $n=3$ and $m=2$, we have $\ast(a\wedge b)=a\times b$ for $a,b\in\RR^3$. Linearity of $\encvol$ directly implies that 
\begin{equation}\label{eq:enclosedvolumeconv}
\text{$V_h\wto^\ast V$ in $\Radon(\Omega\times G_{n-1,n}^o)$, $\:\supp (V_h)$ compact}\:\:\Longrightarrow\:\:\encvol(V_h)\to \encvol(V).
\end{equation}

Similarly to an integral varifold, an {\it oriented integral $m$-varifold} in $\Omega$ is defined by 
$$
V(\psi):=\int_M\left(\psi(x,\xi(x))\theta_+(x)+\psi(x,-\xi(x))\theta_-(x)\right)\dint\Haus^m(x) \quad\text{for all}\quad\psi\in C_c(\Omega\times G_{m,n}^o),
$$
where $M\sbs\Omega$ is (countably) $m$-rectifiable, $\theta_\pm\colon M\to\NN_0$, called {\it multiplicities}, are locally summable w.r.t.\ $\Haus^m\llc M$, and $\xi$ is the orientation as for a rectifiable current (cf.\ Section \ref{sec:currents}). We will use the notation $V=v^o(M,\xi,\theta_+,\theta_-)$ to indicate such a varifold and we denote by $IV_m^o(\Omega)$ the space of all oriented integral $m$-varifolds in $\Omega$. The unoriented counterpart of $V$ is given by $q_\sharp V=v(M,\theta_++\theta_-)\in IV_m(\Omega)$. Moreover, one can associate to $V$ the integer multiplicity $m$-rectifiable current $\curr{V}:=[M,\xi,\theta_+-\theta_-]$. As $\theta_\pm$ take values in the nonnegative integers, the current mass is bounded by the mass of $V$: 
\begin{equation}\label{eq:massbound}
\mathbb{M}_{\curr{V}}(\Omega)=\int_M|\theta_+-\theta_-|\dint\Haus^m
\leq\int_M(\theta_++\theta_-)\dint\Haus^m=\mu_{q_\sharp V}(\Omega)=\mu_V(\Omega).
\end{equation} 

In \cite{Hutchinson:86}, Hutchinson presented the following compactness (closedness) result. 

\begin{theorem}[Compactness for oriented varifolds]\label{thm:Hutchinson_Thm3} 
Let $(V_h)$ be a sequence in $IV_m^o(\Omega)$ with 
$$
\mu_{V_h}(\Omega)+\|\delta (q_\sharp V_h)\|(\Omega)+\MM_{\d\curr{V_h}}(\Omega)\le c
$$
for some $c>0$. Then there exists $V\in IV_m^o(\Omega)$, such that up to a subsequence (not relabeled), $V_h \rightharpoonup^* V$ in $\Radon(\Omega\times G_{m,n}^o)$.
\end{theorem}

We eventually introduce the class of {\it oriented curvature $m$-varifolds},
\begin{equation}\label{eq:AVo}
AV_m^o(\Omega):=\{V \in IV_m^o(\Omega) : q_\sharp V \in AV_m(\Omega)\}.
\end{equation}
Given $V_h,V\in AV_m^o(\Omega)$ and $h\to\infty$, we say that 
$$
V_h\wto^\ast V\:\:\text{in}\:\:AV_m^o(\Omega)
$$ 
if $V_h\wto^\ast V \:\:\text{in}\:\:\Radon(\Omega\times G_{m,n}^o)$ and $q_\sharp V_{h}\wto^\ast q_\sharp V$ in $AV_m(\Omega)$.

If  $M\sbs\Omega$ is a smooth hypersurface of $\RR^n$ with unit normal $\nu \colon M \to \mathbb S^{n-1}$ 
we associate to it the canonical oriented curvature $(n-1)$-varifold 
\begin{equation}
\label{eq:VM}
V_M:=v^o(M,\xi,1,0)\in AV_{n-1}^o(\Omega), \quad \nu=\ast\xi.
\end{equation}
Using smooth differential geometry (see for instance Hutchinson \cite{Hutchinson:86}) it can be verified that 
$$
(A^{q_\sharp V_M})_{ijk} = A_{ijk} := P_{il}\partial_lP_{jk}\quad\text{and} \quad \bar H^{q_\sharp V_M}= \bar H :=  H\nu,
$$
where $P=\mathbb{I}-\nu\otimes \nu$ is the tangential projection matrix and $H$ is the scalar mean curvature of $M$. In particular, in case of smooth surfaces $M$ in $\RR^3$ with $\kappa_1$, $\kappa_2$ denoting their principal curvatures, we have $H=\kappa_1+\kappa_2$ and $K=\kappa_1\kappa_2$ defines the Gauss curvature of $M$.  Since  $\frac{1}{2}(\kappa_1+\kappa_2)^2-\frac{1}{2}(\kappa_1^2+\kappa_2^2)=\kappa_1\kappa_2$,  we have 
$$
K=\frac{1}{2}|\bar H|^2-\frac{1}{2}|\II|^2
=\frac{1}{2}|\bar H|^2-\frac{1}{4}|A|^2.
$$
For the second equality we employed $|A|^2=2|\II|^2$, which is a consequence of $A_{ijk}=\II^j_{ik}+\II^k_{ij}$ and the mapping properties of 
the second fundamental form $\II=-(\nabla^P\nu)\otimes\nu$,  i.e.\ in components,  $\II_{ij}^k=-(\d^P_j\nu_i)\nu_k$ with the tangential partial derivative $\d_j^P=P_{jk}\d_k$.

\section{The Canham-Helfrich functional for oriented curvature varifolds}\label{sec:formulation}

We introduce a formulation of the Canham-Helfrich energy in the framework of oriented  curvature varifolds,  see \eqref{eq:AVo}.  Let $M$ be a compact smooth surface embedded in an open subset $\Omega\sbs\RR^3$ and assume for the time being that $V\in AV_2^o(\Omega)$ is the associated varifold \eqref{eq:VM}. Then the quantities 
\[
\bar H_i=H\,\nu_i=\sum_{j=1}^3 A_{jij}\:\: (i=1,2,3), \quad  K=\frac{1}{2} |\bar H|^2 -\frac{1}{4}|A|^2,\quad \text{and}\quad \nu=\ast\xi
\]
are defined for $V$ and one can rewrite the integrand of the functional \eqref{CH} as
\begin{align*}
\frac{\beta}{2}\,(H-H_0)^2+\gamma\, K
&=\frac{\beta}{2}\,|\bar{H}-(\ast\xi) H_0|^2+\frac{\gamma}{2}|\bar{H}|^2-\frac{\gamma}{4}|A|^2\nn\\
&=\sum_{i=1}^3\left(\frac{\beta}{2}\,\left(\bar{H}_i-(\ast\xi)_i H_0\right)^2+\frac{\gamma}{2}\bar{H}_i^2-\frac{\gamma}{4}\sum_{j,k=1}^3 A_{ijk}^2\right).
\end{align*} 
The latter suggests to define the {\it generalized Canham-Helfrich functional} as 
\begin{equation}\label{eq:FCH}
\FCH\colon AV_2^o(\Omega)\to\RR,\qquad 
\FCH(V):=\int_{\Omega\times G_{2,3}^o} f_{\CH}\big(\xi,A^{q_\sharp V}(x,q(\xi))\big)\dint V(x,\xi),
\end{equation}
where the density $f_\CH\colon G_{2,3}^o\times\RR^{3\times 3\times 3}\to\RR$ is given by 
$$
f_\CH(\xi,A):=\sum_{i=1}^3\left(\frac{\beta}{2}\left(\Bigg(\sum_{j=1}^3A_{jij}\Bigg)-(\ast\,\xi)_iH_0\right)^2
+\frac{\gamma}{2}\Bigg(\sum_{j=1}^3A_{jij}\Bigg)^2
-\frac{\gamma}{4}\sum_{j,k=1}^3 A_{ijk}^2\right).
$$
Recall from Section \ref{sec:varifolds} that the elements of $AV_2^o(\Omega)$  are oriented integral $2$-varifolds  
$$
V=v^o(M,\xi,\theta_+,\theta_-)\in IV_2^o(\Omega)
$$
in the sense of Hutchinson \cite{Hutchinson:86}, whose unoriented counterparts   
$$
q_\sharp V=v(M,\theta_++\theta_-)\in AV_2(\Omega)
$$
are curvature $2$-varifolds in the sense of Mantegazza \cite{Mantegazza:96}.
Let 
$$
P\colon M\to G_{2,3}\quad\text{and}\quad \pm\xi\colon M\to G_{2,3}^o
$$ 
denote the mappings that assign the approximate tangent space $P(x):=T_xM$ and its oriented counterparts $\pm\xi(x)$ respectively to $(\Haus^2\llc M)$-a.e.\ $x\in M$. Then $P$ and $\xi$ are related by the covering map $q\colon G_{2,3}^o\to G_{2,3}$, namely $P(x)=q(\xi(x))=q(-\xi(x))$. It follows that the generalized Canham-Helfrich functional \eqref{eq:FCH} takes the form 
\begin{align}
\FCH(V)&=\int_M\Big(f_{\CH}(\xi(x),A^{q_\sharp V}(x,P(x)))\,\theta_+(x)
+f_{\CH}(-\xi(x),A^{q_\sharp V}(x,P(x)))\,\theta_-(x)\Big)\dint\Haus^2(x) \nonumber\\
&=\sum_{i=1}^3\int_M\Bigg(
\frac{\beta}{2}\Big({\bar H}_i^{q_\sharp V}(\cdot,P) -(\ast\,\xi)_iH_0\Big)^2\,\theta_+
+\frac{\beta}{2}\Big({\bar H}_i^{q_\sharp V}(\cdot,P)+(\ast\,\xi)_iH_0\Big)^2\,\theta_-\nonumber\\
&\qquad\qquad\qquad +\bigg(\frac{\gamma}{2}\big({\bar H}_i^{q_\sharp V}(\cdot,P)\big)^2
-\frac{\gamma}{4}\sum_{j,k=1}^3 \big(A_{ijk}^{q_\sharp V}(\cdot,P)\big)^2\bigg)(\theta_++\theta_-)\Bigg)\dint\Haus^2,\label{eq:FCHlong}
\end{align}
where for $\Haus^2$-a.e.\ $x\in M$ and $i=1,2,3$,
\begin{equation*}
{\bar H}_i^{q_\sharp V}(x,P(x))=\sum_{j=1}^3A_{jij}^{q_\sharp V}(x,P(x)).
\end{equation*}
We want to investigate now coercivity and lower semicontinuity of $\FCH$ with respect to the convergence in $AV_2^o(\Omega)$. Let us point out that the lower semicontinuity of functionals of Canham-Helfrich type is a delicate issue, for counterexamples in specific situations are known \cite[Rem. (ii), p.\ 550]{GrosseBrauckmann:93}. 

Firstly, we first establish conditions on the material parameters $\beta$ and $\gamma$ that guarantee strict convexity of the energy density in the curvature variables. 

\begin{prop}[Strict convexity of $f_{\CH}$]\label{prop:convexity}
If $\beta$ and $\gamma$ satisfy the relation 
\begin{equation}\label{eq:bounds}
-\frac{6}{5}\,\beta<\gamma<0,
\end{equation}
then 
$f_\CH(\xi,\cdot)\colon\RR^{3\times 3\times 3}\to\RR$ is strictly convex for all $\xi\in G_{2,3}^o$. Moreover, there exist $c_1,c_2>0$ such that
\begin{equation}\label{eq:estimate}
\|A^{q_\sharp V}\|^2_{L^2_{q_\sharp V}(\Omega\times G_{2,3})}\leq c_1\left(\FCH(V)+c_2\,\mu_{V}(\Omega)\right) \quad\text{for all}\quad V\in AV_2^o(\Omega).
\end{equation}
\end{prop}

\begin{proof} We split $f_\CH$ in the sum of quadratic and linear terms in $A$. By using the fact that $\sum_{i,j,k=1}^3 A_{ijk}^2=\sum_{i,j,k=1}^3A_{jik}^2$ we write
\begin{align*}
&f_\CH(\xi,A)
=\sum_{i=1}^3\Bigg(\frac{\beta}{2}\bigg(\Big(\sum_{j=1}^3A_{jij}\Big)-(\ast\,\xi)_iH_0\bigg)^2
+\frac{\gamma}{2}\Big(\sum_{j=1}^3A_{jij}\Big)^2
-\frac{\gamma}{4}\sum_{j,k=1}^3 A_{ijk}^2\Bigg)\\
&=\sum_{i=1}^3\underbrace{\Bigg(\frac{\beta+\gamma}{2}\Bigg(\sum_{j=1}^3A_{jij}\Bigg)^2
-\frac{\gamma}{4}\sum_{j,k=1}^3 A_{jik}^2\Bigg)}_{=:\,f(a^i)}
+\sum_{i=1}^3\underbrace{\Bigg(-\beta\Big(\sum_{j=1}^3A_{jij}\Big)(\ast\,\xi)_iH_0+\frac{\beta}{2}(\ast\,\xi)_i^2H_0^2\Bigg)}_{=:\,l_i(\xi,a^i)}\\
&=\sum_{i=1}^3 \left(f(a^i)+l_i(\xi,a^i)\right).
\end{align*}
For the last equality we introduced
$$
a^i:=(a_{11}^i, a_{22}^i, a_{33}^i, a_{12}^i, a_{13}^i, a_{23}^i, a_{21}^i, a_{31}^i, a_{32}^i)\in\RR^{9}
$$
with $a_{jk}^i:=A_{jik}$. Let us now check that $f: \RR^9 \to \RR$ is strictly convex. Indeed, one has that  
\begin{align*}
f(a)&=\frac{\beta+\gamma}{2}\Big(\sum_{j=1}^3a_{jj}\Big)^2-\frac{\gamma}{4}\sum_{j,k=1}^3 a_{jk}^2\\
&=\frac{\beta+\gamma}{2}\left(a_{11}^2+a_{22}^2+a_{33}^2+2a_{11}a_{22}+2a_{11}a_{33}+2a_{22}a_{33}\right)\\
&\qquad -\frac{\gamma}{4}\left(a_{11}^2+a_{22}^2+a_{33}^2+a_{12}^2+a_{13}^2+a_{23}^2+a_{21}^2+a_{31}^2+a_{32}^2\right)\\
&=\frac{2\beta+\gamma}{4}\left(a_{11}^2+a_{22}^2+a_{33}^2\right)+(\beta+\gamma)(a_{11}a_{22}+a_{11}a_{33}+a_{22}a_{33})\\
&\qquad -\frac{\gamma}{4}\left(a_{12}^2+a_{13}^2+a_{23}^2+a_{21}^2+a_{31}^2+a_{32}^2\right).
\end{align*}	
By computing partial derivatives one obtains 
\begin{align*}
\frac{\d f}{\d a_{11}}&=\frac{2\beta+\gamma}{2}\,a_{11}+(\beta+\gamma)\left(a_{22}+a_{33}\right),
\qquad
\frac{\d f}{\d a_{12}}=-\frac{\gamma}{2}\,a_{12}, \\ 
\frac{\d^2 f}{\d a_{11}^2}&=\frac{2\beta+\gamma}{2},\qquad 
\frac{\d^2 f}{\d a_{11}\d a_{22}}=\beta+\gamma,\qquad\quad
\frac{\d^2 f}{\d a_{12}^2}=-\frac{\gamma}{2}  
\end{align*}
and analogously for the other components. The Hessian
$D^2 f(a)\in\RR^{9\times 9}$ then reads 
$$
D^2 f(a)=\begin{pmatrix}
\frac{2\beta+\gamma}{2} & \beta+\gamma & \beta+\gamma & 0 & 0 & 0 & 0 & 0 & 0\\
\beta+\gamma &\frac{2\beta+\gamma}{2} & \beta+\gamma & 0 & 0 & 0 & 0 & 0 & 0\\
\beta+\gamma & \beta+\gamma &\frac{2\beta+\gamma}{2} & 0 & 0 & 0 & 0 & 0 & 0\\
0 & 0 & 0 & -\frac{\gamma}{2} & 0 & 0 & 0 & 0 & 0\\
0 & 0 & 0 & 0 & -\frac{\gamma}{2} & 0 & 0 & 0 & 0\\
0 & 0 & 0 & 0 & 0 & -\frac{\gamma}{2} & 0 & 0 & 0\\
0 & 0 & 0 & 0 & 0 & 0 & -\frac{\gamma}{2} & 0 & 0\\
0 & 0 & 0 & 0 & 0 & 0 & 0 & -\frac{\gamma}{2} & 0\\
0 & 0 & 0 & 0 & 0 & 0 & 0 & 0 & -\frac{\gamma}{2}
\end{pmatrix}. 
$$
Its eigenvalues are  
$$
\lambda_1=\frac{1}{2}(6\beta+5\gamma)\qquad\text{and}\qquad
\lambda_{2} = \lambda_3 = \dots = \lambda_9 =-\frac{\gamma}{2}.
$$
Thus the Hessian is positive definite ($\lambda_k>0$ for $1\leq k\leq 9$) if and only if $6\beta+5\gamma>0$ and $\gamma<0$, which is equivalent to $-6\beta<5\gamma<0$, namely \eqref{eq:bounds}. The integrand $f_\CH(\xi,\cdot)$ is hence a strictly convex second-order polynomial for all $\xi\in G_{2,3}^o$. Moreover, $\xi$ only occurs explicitly in the term $\sum_{i=1}^3l_i(\xi,a^i)$, which is smooth, being a second-order polynomial in the components of $\xi$. Consequently, there exist a constant $c_1>0$ and continuous maps $c_2^\pm\colon G_{2,3}^o\to(0,\infty)$,  such that
$$
f_\CH(\xi,A)\geq  \frac{1}{c_1}|A|^2-c_2^+(\xi),\qquad f_\CH(-\xi,A)\geq  \frac{1}{c_1}|A|^2-c_2^-(\xi).
$$
Since $\theta_\pm(x)\geq 0$ for $\Haus^2$-a.e.\ $x\in M$, from \eqref{eq:FCHlong} we then get
\begin{align*}
f_\CH(\xi,A)\theta_++f_\CH(-\xi,A)\theta_-
&\geq \left(\frac{1}{c_1}|A|^2-c_2^+(\xi)\right)\theta_++\left(\frac{1}{c_1}|A|^2-c_2^-(\xi)\right)\theta_-\\
&=\frac{1}{c_1}|A|^2(\theta_++\theta_-)-(c_2^+(\xi)\theta_++c_2^-(\xi)\theta_-)\nn\\
&\geq\frac{1}{c_1}|A|^2(\theta_++\theta_-)-c_2(\theta_++\theta_-)
=\left(\frac{1}{c_1}|A|^2-c_2\right)(\theta_++\theta_-)\nn
\end{align*}
where $ c_2:=\max_{\xi\in G_{2,3}^o}\max\{c_2^+(\xi),c_2^-(\xi)\}$.
Therefore, we conclude that
\begin{align*}
\FCH(V)&\geq\int_M\left(\frac{1}{c_1}|A^{q_\sharp V}(x,q(\xi(x)))|^2-c_2\right)(\theta_++\theta_-)(x)\dint\Haus^2(x)\nn\\
&=\int_{\Omega\times G_{2,3}} \left(\frac{1}{c_1}|A^{q_\sharp V}(x,P)|^2-c_2\right)\dint(q_\sharp V)(x,P)\\
&=\frac{1}{c_1}\|A^{q_\sharp V}\|^2_{L^2_{q_\sharp V}(\Omega\times G_{2,3})}-c_2\,\mu_V(\Omega)
\end{align*}
and the estimate \eqref{eq:estimate} is proved.
\end{proof}

Secondly, we state the main result of this section:
 
\begin{theorem}[Lower semicontinuity]\label{thm:lsc}
Assume \eqref{eq:bounds}. Then $\FCH$ is lower semicontinuous with respect to the convergence in $AV_2^o(\Omega)$.
\end{theorem}

\begin{proof} 
Let $V_h\wto^\ast V$ in $AV_2^o(\Omega)$ and assume with no loss of generality that 
$$
\sup_h\FCH(V_h)<\infty. 
$$
We have that $V_h\wto^\ast V$ converges as oriented varifolds in $\Radon(\Omega\times G_{2,3}^o)$ and  $(q_\sharp V_h,A^{q_\sharp V_h})\wto^\ast (q_\sharp V,A^{q_\sharp V})$ converges as measure-function pairs over $\Omega\times G_{2,3}$. We observe that 
$$
\FCH(V)=\int_{\Omega\times G_{2,3}^o} f_{\CH}(\xi,A^{q_\sharp V}(x,q(\xi)))\dint V(x,\xi)
= \int_{\Omega \times G^o_{2,3}}f_{\CH} (\xi, B(x,\xi)) \dint V(x,\xi)
$$
where the function $B\colon\Omega\times G_{2,3}^o\to\RR^{3\times 3 \times 3}$ is given by
$$
B(x,\xi):=A^{q_\sharp V}(x,q(\xi)).
$$
We now pass to the limit with respect to the measure-function pair $(V,B)$ over $\Omega\times G_{2,3}^o$ with values in $\RR^{3\times 3 \times 3}$. The symmetries 
$$
B(x,\xi)=A^{q_\sharp V}(x,q(\xi))=A^{q_\sharp V}(x,q(-\xi))=B(x,-\xi)
$$
and
$$
B_h(x,\xi):=A^{q_\sharp V_h}(x,q(\xi))=A^{q_\sharp V_h}(x,q(-\xi))=B_h(x,-\xi)
$$
allow us to infer that
$$
(V_h,B_h)\wto^\ast (V,B)
$$ 
in the sense of measure-function  pair  convergence over $\Omega\times G_{2,3}^o$.  The integrand  $f_\CH$ is continuous on $G_{2,3}^o\times\RR^{3\times 3 \times 3}$ and a strictly convex quadratic function in the second variable  by Proposition \ref{prop:convexity}.  In particular, there exist $c_1, c_2>0$ such that for all $\xi\in G_{2,3}^o$
$$
f_\CH(\xi,B)\geq\frac{1}{c_1}|B|^2-c_2.
$$
Therefore, Theorem \ref{thm:mfp} (ii) applies and we obtain
\begin{align*}
\FCH(V)&=\int_{\Omega\times G_{2,3}^o} f_{\CH}(\xi,B(x,\xi))\dint V(x,\xi)\\
& \leq \liminf_{h\to\infty}\int_{\Omega\times G_{2,3}^o} f_{\CH}(\xi,B_h(x,\xi))\dint V_h(x,\xi) \\
&=\liminf_{h\to\infty}\FCH(V_{h}),
\end{align*}
which concludes the proof.
\end{proof}

\section{Existence of minimizers}\label{sec:existence}

In order to establish the existence of minimizers of $\FCH$ among varifolds with fixed total mass, we assume that the spatial support of the admissible varifolds is contained in a fixed compact set. Additionally, to guarantee that weak limits are still orientable, we require that the mass of the boundary current of the admissible varifolds is bounded by a fixed constant.

\begin{lemma}[Closedness of admissible varifolds]\label{lem:Aclosed}
Let $K\subset\subset\Omega$ have nonempty interior and let $m^o>0$. Then the class  
\begin{equation}\label{eq:admissible}
\cA:=\left\{V\in AV_2^o(\Omega):\:\supp\,\mu_V\sbs K,\:\:\mathbb{M}_{\d\curr{V}}(\Omega)\leq m^o\right\}
\end{equation}
is closed in $AV_2^o(\Omega)$.
\end{lemma}
\begin{proof}
Consider $V_h\in\cA$  with $V_{h}\wto^\ast V$ in $AV_2^o(\Omega)$ \eqref{eq:AVo} as $h\to\infty$. With the convergence of the mass measures, $\mu_{V_h}\wto^\ast\mu_V$ in $\Radon(\Omega)$, standard results from measure theory then give
$$ 
V_{h}\wto^\ast V\:\:\text{in}\:\:AV_2^o(\Omega),\quad \supp\,\mu_{V_h}\sbs K \quad\Longrightarrow\quad \supp\,\mu_V\sbs K.
$$
Moreover, the $2$-currents $\curr{V_h}$ and $\curr{V}$ corresponding to $V_h$ and $V$ respectively are $2$-rectifiable with integer multiplicity. We have
$$
V_{h}\wto^\ast V\:\:\text{in}\:\:AV_2^o(\Omega)
\quad\Longrightarrow\quad
\curr{V_{h}}\wto^\ast \curr{V}\:\:\text{in}\:\:\mathcal D_2(\Omega)
\quad\Longrightarrow\quad
\d\curr{V_{h}}\wto^\ast \d\curr{V}\:\:\text{in}\:\:\mathcal D_1(\Omega).
$$
The masses $\mathbb{M}_{\curr{V_h}}(\Omega)$ are bounded by virtue of estimate \eqref{eq:massbound}. With $\mathbb{M}_{\d\curr{V_h}}(\Omega)\leq m^o$, the lower semicontinuity of the mass of currents (Theorem \ref{thm:FF}) then yields  
$$
\mathbb{M}_{\d\curr{V}}(\Omega)\leq \liminf_{h\to\infty}\mathbb{M}_{\d\curr{V_h}}(\Omega)\leq m^o,
$$
completing the proof.
\end{proof}

\subsection{Single-phase membranes}

In this paragraph we present an existence result of varifold-minimizers for the single-phase Canham-Helfrich energy \eqref{eq:FCH}, including the constraints of fixed surface area and enclosed volume, 
\begin{equation}\label{eq:me}
m>0\quad\text{and}\quad e>0.
\end{equation}
In order not to rule out embedded smooth solutions, we further assume that $m$ and $e$ satisfy the isoperimetric inequality in $\RR^3$,
$$
(6\,\sqrt{\pi}\,e)^{1/3}\leq m^{1/2}.
$$
Equality among smooth closed  embedded  surfaces is achieved  by  a sphere.

\begin{theorem}[Existence]\label{thm:single}
Assume \eqref{eq:bounds} and \eqref{eq:me}. Then there exists a solution of
$$
\min\{\FCH(V):\:V\in\cA,\, \mu_V(\Omega)=m,\, \d(q_\sharp V)=0,\,\encvol(V)=e\}.
$$
\end{theorem}

\begin{proof} As varifolds $V\in \cA$ with $\mu_V(\Omega) = m$, $\d(q_\sharp V)=0$, and $\encvol(V)=e$ exist (think of ellipsoids), one can find a minimizing sequence $(V_h)$  for $\FCH$, i.e.,  
$$
\lim_{h\to\infty}\FCH(V_h)=\inf\{\FCH(V):\,V\in\cA,\, \mu_V(\Omega)=m,\, \d(q_\sharp V)=0,\,\encvol(V)=e\},
$$
where $V_h\in\cA$ and satisfies $\mu_{V_h}(\Omega)=m$, $\d(q_\sharp V_h)=0$, $\encvol(V_h)=e$. By Proposition \ref{prop:convexity},
$$
\|A^{q_\sharp V_h}\|^2_{L^2_{q_\sharp V_h}(\Omega\times G_{2,3})}\leq c'\left( \FCH(V_h)+\mu_{V_h}(\Omega)\right),
$$
for some $c'>0$, which with \eqref{eq:firstvarbound} yields
\begin{align}
\|\delta(q_\sharp V_h)\|(\Omega)
&\leq \sqrt{2\,\mu_{V_h}(\Omega)}\, \|A^{q_\sharp V_h}\|_{L^2_{q_\sharp V_h}(\Omega\times G_{2,3})}
+\underbrace{\|\d (q_\sharp V_h)\|(\Omega\times G_{2,3})}_{=\:0}\nn\\
&\leq c\sqrt{\mu_{V_h}(\Omega)}\,\sqrt{\FCH(V_h)+\mu_{V_h}(\Omega)}
\label{eq:firstvarestimate}
\end{align}
for some $c>0$. From $\mu_{V_h}(\Omega)=m$ and the minimality of $V_h$ we thus obtain 
$$
\|\delta(q_\sharp V_h)\|(\Omega)\leq c\sqrt{m}\,\sqrt{\FCH(V_h)+m}<\infty
\quad \text{uniformly in $h$}.
$$ 
By the compactness Theorem \ref{thm:Hutchinson_Thm3},
$$
\mu_{V_{h}}(\Omega)+\|\delta(q_\sharp V_{h})\|(\Omega)+\mathbb{M}_{\d\curr{V_{h}}}(\Omega) <\infty
\:\Longrightarrow\: \exists\:\: V_{h'}\wto^\ast V\:\:\text{in}\:\:\Radon(\Omega\times G_{2,3}^o),\:\: V \in IV_2^o(\Omega).
$$
Consequently, by Lemma \ref{lem:orconv}, the subsequence $(V_{h'})$ also satisfies
$$
q_\sharp V_{h'}\wto^\ast q_\sharp V\:\:\text{in}\:\:\Radon(\Omega\times G_{2,3}),\quad q_\sharp V\in IV_2(\Omega).
$$
By the compactness Theorem \ref{thm:Mantegazza_Thm6},
$$
\mu_{q_\sharp V_{h'}}(\Omega)+\|A^{q_\sharp V_{h'}}\|^2_{L^2_{q_\sharp V_{h'}}(\Omega\times G_{2,3})}
+\underbrace{\|\d (q_\sharp V_{h'})\|(\Omega\times G_{2,3})}_{=0}<\infty
\:\Longrightarrow\: q_\sharp V_{h'}\wto^\ast \widetilde{V}\:\:\text{in}\:\: AV_2(\Omega).
$$
Since convergence in $AV_2(\Omega)$ and $IV_2^o(\Omega)$ imply convergence as Radon measures  $\Radon(\Omega\times G_{2,3})$,  uniqueness of limits yields that $
\widetilde{V}=q_\sharp V$. Consequently, the minimizing sequence $V_h$ has a subsequence which converges in $\cA$, i.e.\
$$
V_{h'}\wto^\ast V\quad\text{in}\quad AV_2^o(\Omega).
$$
The limit $V$ also satisfies the constraints on admissibility, mass, zero boundary, and enclosed volume. Indeed, $V\in\cA$ follows from closedness of $\cA$ established in Lemma \ref{lem:Aclosed}. The compact support of admissible varifolds allows us to infer $\mu_{V}(\Omega)=m$ via results from measure theory. The  condition on the boundary measure  $\d(q_\sharp V)=0$ follows from Mantegazza's Theorem \ref{thm:Mantegazza_Thm6}, according to which $\d(q_\sharp V_{h'})\wto^\ast \d(q_\sharp V)$ in $\Radon^3(\Omega\times G_{2,3})$. Finally, $\encvol(V)=e$ holds by virtue of \eqref{eq:enclosedvolumeconv}. Moreover, Theorem \ref{thm:lsc} shows that $\FCH$ is lower semicontinuous on $AV_2^o(\Omega)$. Therefore, by the direct method, the limit $V\in\cA$ is a minimizer.
\end{proof} 

\subsection{Sharp-interface multiphase membranes}

We introduce the two-phase energy 
$$
{\FCH}_{,2}(V^{\one},V^{\two}):=\FCH^{\one}(V^{\one})+\FCH^{\two}(V^{\two})+\sigma^\one\|\d(q_\sharp V^\one)\|(\Omega\times G_{2,3})+\sigma^\two\|\d(q_\sharp V^\two)\|(\Omega\times G_{2,3})
$$
for $V^\one,V^\two\in {\mathcal A}$ and some $\sigma^\one,\sigma^\two>0$. Here, $\FCH^{\one}$ and $\FCH^{\two}$ are defined as in \eqref{eq:FCHlong} with parameters $ \beta^\one, \, \gamma^\one$ and $\beta^\two,\,\gamma^\two$, respectively. The two varifolds $V^\one$ and $V^\two$ correspond to the individual phases.

In order to model sharp interfaces, the spatial supports of $V^\one$ and $V^\two$ are required  not to overlap on $\Haus^2$-nonnegligible sets. According to definition \eqref{eq:admissible}, admissible varifolds $V_i\in\cA$ ($i=1,2$) are integral, whence their weight measures are of the form 
$$
\mu_i:=\mu_{V^{i}}=\theta_i\,\Haus^2\llc M_i\in\Radon(\Omega)
$$ 
with $M_i\sbs\Omega$ countably $2$-rectifiable, $\Haus^2$-measurable and with positive, integer-valued multiplicity $\theta_i\in L^1_{\loc;\Haus^2\llc M_i}(M_i;\NN)$. Overlap is quantified in terms of the scaling of the product-mass $\mu_{{\one}}\otimes\mu_{{\two}}\in\Radon(\Omega\times\Omega)$ evaluated on the open $\veps$-neighborhood 
$$
D_{\Omega,\veps}:=\{(x,y)\in \Omega\times\Omega:\,|x-y|<\veps\}
$$
of the diagonal in $\Omega\times\Omega$: We say that the varifolds $V_\one$, $V_\two\in\cA$  satisfy the {\em no-overlap condition} with $\veps_0>0$, if
\begin{equation}\label{eq:nooverlap}
(\mu_{\one}\otimes\mu_{\two})(D_{\Omega,\veps})\leq \veps^3/\veps_0\quad\text{for all}\quad \veps\in(0,\veps_0).
\end{equation}
The next lemma shows that this condition rules out $\Haus^2$-overlap of $M_\one$ and $M_\two$, still allowing $\Haus^1$- or $\Haus^0$-overlap. 

\begin{lemma}[{{No-Overlap}}]\label{lem:nooverlap}
Let $\mu_i=\theta_i\,\Haus^2\llc M_i\in\Radon(\Omega)$ be as above and have fixed mass $m^i=\mu_i(\Omega)>0$ for $i=1,2$. Then, by letting $S_\veps:=\{x\in M_\one:\:\dist(x,M_\two)<\veps\}$ we have that 
\begin{equation}\label{eq:nooverlapformula}
(\mu_{\one}\otimes\mu_{\two})(D_{\Omega,\veps})=\int_{S_\veps}\mu_{\two}(B_\veps(x))\dint \mu_{\one}(x)
\end{equation}
and the following statements hold:

\begin{enumerate}[{(i)}]
\item If there exists $c>0$ such that for all $\veps\in (0,\veps_0)$, $\mu_\two(B_\veps(x))\geq c\,\veps^2$ for all $x\in M_\two$ and $\Haus^2(M_\one\cap M_\two)>0$, then condition \eqref{eq:nooverlap} is not satisfied. 
\item If there exists $C>0$ such that for all $\veps\in (0,\veps_0)$, $\mu_\two(B_\veps(x))\leq C\,\veps^2$ for all $x\in S_\veps$, $\mu_\one(S_\veps)\leq\veps/(C\veps_0)$, and $0<\Haus^m(M_1\cap M_2)<\infty$ with $ m\in[0,1]$, then \eqref{eq:nooverlap} is fulfilled.
\end{enumerate} 
Analogous statements hold if $\mu_\one$ and $\mu_\two$ are interchanged.
\end{lemma}

See Figures \ref{fig:H2}, \ref{fig:H1}, and \ref{fig:H0} for some illustration of the statement above and its proof.

\begin{figure}
		\vspace{0.5em}
	\includegraphics[width=0.55\textwidth]{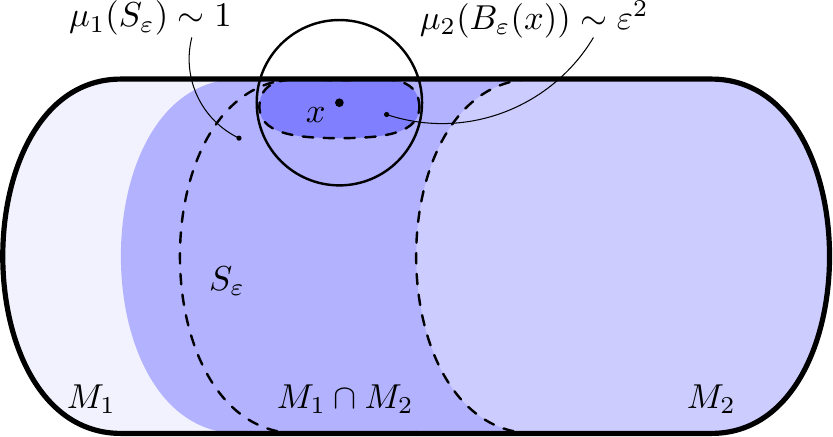}
	\caption{$\Haus^2$-overlap, Lemma \ref{lem:nooverlap}{\em (i)}.}
	\label{fig:H2}
\end{figure}

\begin{figure}
	\vspace{2em}
	\includegraphics[width=0.55\textwidth]{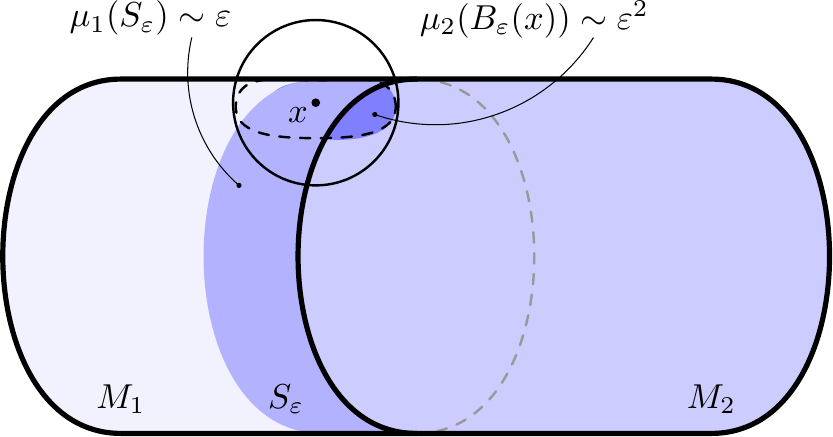}
	\caption{$\Haus^1$-overlap, Lemma \ref{lem:nooverlap}{\em (ii)} with $m=1$.}
	\label{fig:H1}
\end{figure}

\begin{figure}
	\vspace{2em}
	\includegraphics[width=0.55\textwidth]{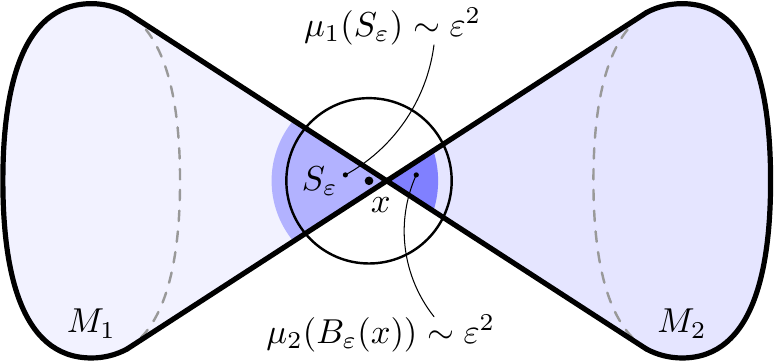}
	\caption{$\Haus^0$-overlap, Lemma \ref{lem:nooverlap}{\em (ii)} with $m=0$.}
	\label{fig:H0}
\end{figure}

\begin{proof} Let $\veps>0$.  The fixed total masses imply that $\mu_i/m^i\in\Radon(\Omega)$ 
and $(\mu_{\one}\otimes\mu_{\two})/(m^{\one}m^{\two})\in\Radon(\Omega\times\Omega)$ 
are probability measures. By disintegration for probability measures (see \cite[Thm.\ 2.28, p.\ 57]{AFP:00} or \cite[Thm.\ 5.3.1, p.\ 121]{AGS:08}) and $|x-y|<\veps$ being equivalent to $y\in B_\veps(x)$,
we obtain
$$
(\mu_{\one}\otimes\mu_{\two})(D_{\Omega,\veps})
=\int_{D_{\Omega,\veps}}\dint (\mu_{\one}\otimes\mu_{\two})
=\int_{\Omega}\mu_{\two}(B_\veps(x))\dint \mu_{\one}(x)
=\int_{M_\one}\mu_{\two}(B_\veps(x))\dint \mu_{\one}(x).
$$
It holds that 
$$
\mu_{\two}(B_\veps(x))=0\quad\text{for all}\quad x\in \Omega,\:\dist(x,M_2)\geq\veps.
$$
With $S_\veps=\{x\in M_\one:\:\dist(x,M_\two)<\veps\}$, this implies that $\mu_{\two}(B_\veps(x))=0$ for all $x\in M_\one\setminus S_\veps$, which shows \eqref{eq:nooverlapformula}. 

To prove statement {\em (i)}, assume that $\Haus^2(M_\one\cap M_\two)>0$. With 
$S_\veps\supseteq M_\one\cap M_\two$ and the lower bound $\mu_\two(B_\veps(x))\geq c\,\veps^2$ for  $x\in M_\two$ we arrive at the estimate
\begin{align*}
(\mu_{\one}\otimes\mu_{\two})(D_{\Omega,\veps})
=\int_{S_\veps}\mu_{\two}(B_\veps(x))\dint \mu_{\one}(x)
&\geq \int_{M_\one\cap M_\two}\mu_{\two}(B_\veps(x))\dint \mu_{\one}(x)\\
&\geq c\,\veps^2\int_{M_\one\cap M_\two}\underbrace{\theta_\one(x)}_{\geq\, 1}\dint\Haus^2(x)
\geq c\,\veps^2 \Haus^2(M_\one\cap M_\two).
\end{align*}
This gives a contradiction to \eqref{eq:nooverlap} for any $\veps_0>0$, as the choice 
$$
0<\veps<\min\{1,\,\veps_0,\,\veps_0\,c\,\Haus^2(M_1\cap M_2)\}
$$ 
yields $c\,\Haus^2(M_\one\cap M_\two)>\veps/\veps_0$ and thus 
$(\mu_{\one}\otimes\mu_{\two})(D_{\Omega,\veps})\geq c\,\veps^2 \Haus^2(M_\one\cap M_\two)>\veps^3/\veps_0$.

Under the assumptions in  {\em (ii)},  namely $\mu_\two(B_\veps(x))\leq C\,\veps^2$ for $x\in S_\veps$ and $\mu_\one(S_\veps)\leq\veps/(C\veps_0)$,  we obtain
$$
(\mu_{\one}\otimes\mu_{\two})(D_{\Omega,\veps})
=\int_{S_\veps}\mu_{\two}(B_\veps(x))\dint \mu_{\one}(x)
\leq C\,\veps^2\,\mu_\one(S_\veps)\leq C\,\veps^2\,\veps/(C\veps_0)= \veps^3/\veps_0
$$
so that  \eqref{eq:nooverlap} holds.
\end{proof}

The no-overlap condition \eqref{eq:nooverlap} is closed with respect to oriented curvature varifold convergence, which is crucial for including it as a constraint in the direct method:

\begin{lemma}[Closedness of the no-overlap condition] Let $V^i_h\wto^\ast V^i$ in $AV_2^o(\Omega)$ for $i=\one,\two$. If all $V^\one_h$ and $V^\two_h$ satisfy \eqref{eq:nooverlap} with $\veps_0>0$, then this is also true for the limits $V^\one$ and $V^\two$.
\end{lemma}

\begin{proof} The convergence $V^i_h\wto^\ast V^i$ in $AV_2^o(\Omega)$ implies $\mu_{V^i_h}\wto^\ast\mu_{V^i}$ in $\Radon(\Omega)$.
	The product of Radon measures in $\Radon(\Omega)$ commutes with weak-$\ast$ limits \cite[Thm.\ 8.4.10, p.\ 198]{Bogachev2:07}, namely
	$$
	\mu_{V^{\one}_h}\wto^\ast\mu_{V^{\one}},\quad \mu_{V^{\two}_h}\wto^\ast\mu_{V^{\two}}\quad\text{in}\quad\Radon(\Omega)
	\quad\Longrightarrow\quad
	\mu_{V^{\one}_h}\otimes\mu_{V^{\two}_h}\wto^\ast\mu_{V^{\one}}\otimes\mu_{V^{\two}}\quad\text{in}\quad\Radon(\Omega\times\Omega).
	$$
	Let $\veps\in (0,\veps_0)$. By lower semicontinuity of measures on the open set $D_{\Omega,\veps}$ and by \eqref{eq:nooverlap} for $V^\one_h$ and $V^\two_h$, we obtain
	$$
	(\mu_{V^{\one}}\otimes\mu_{V^{\two}})(D_{\Omega,\veps})
	\leq\liminf_{h\to \infty}\left(\mu_{V^{\one}_h}\otimes\mu_{V^{\two}_h}\right)(D_{\Omega,\veps})\leq \veps^3/\veps_0,
	$$
which proves the claim.
\end{proof}

Note that the no-overlap condition \eqref{eq:nooverlap} implies mutual singularity of varifolds. We have the following. 

\begin{lemma}[Mutual singularity]\label{lem:sing}
	If $V^\one,V^\two\in AV_2(\Omega)$ satisfy \eqref{eq:nooverlap} and, for $i=\one$ or $\two$, there exists $c>0$ with $\mu_i(B_\veps(x))\geq c\,\veps^2$ for all $x\in M_i$, $\veps\in (0,\veps_0)$, then $V^\one\bot V^\two$ in $\Radon(\Omega\times G_{2,3})$. 
\end{lemma}
\begin{proof}
	By Lemma \ref{lem:nooverlap} {\em (i)}, $\mu_i=\mu_{V^{i}}=\theta_i\,\Haus^2\llc M_i$  ($i=\one,\two$) satisfy $\Haus^2(M_\one\cap M_\two)=0$.  Therefore, as $\theta_1\in L^1_{\loc;\Haus^2\llc M_\one}(M_\one;\NN)$ and $M_\two\cap(\Omega\setminus M_\two)=\emptyset$, we have
	$$
	\mu_\one(M_2)=\int_{M_\one\cap M_\two}\theta_\one\,\dint \Haus^2=0
	\qquad\text{and}\qquad
	\mu_\two(\Omega\setminus M_\two)=\int_{M_\two\cap(\Omega\setminus M_\two)}\theta_\two\,\dint \Haus^2=0.
	$$
	Thus $\mu_\one$ and $\mu_\two$ are separated by the Borel set $M_\two$ (and also $M_1$ by an analogous argument). This shows that $\mu_\one\,\bot\,\mu_\two$ in $\Radon(\Omega)$, which yields $V^\one\bot V^\two$ in $\Radon(\Omega\times G_{2,3})$.
\end{proof}
For modeling closed two-phase membranes, the individual phases are required to have matching boundaries, or equivalently, the boundary of the sum (union) of $V^\one$ and $V^\two$ must be empty. 
The sum of mutually singular varifolds features curvature functions that can be expressed in terms of characteristic functions of individual supports; the boundary measure of the sum is simply the sum of the individual boundary measures. Indeed, we have the following.
\begin{lemma}[{{Sum of mutually singular curvature varifolds}}]\label{lem:sum}
If $V^1,V^2\in AV_2(\Omega)$ satisfy $V^\one\bot V^\two$, then $V^1+V^2\in AV_2(\Omega)$ with the curvature functions
\begin{equation}\label{eq:sum_A}
A^{V^\one+V^\two}:=A^{V^\one}\chi_{\supp(V^\one)}+A^{V^\two}\chi_{\supp(V^\two)}
\end{equation}
and the boundary measure 
\begin{equation}\label{eq:sum_bdry}
\d(V^\one+V^\two):=\d V^\one+\d V^\two.
\end{equation}
\end{lemma}

\begin{proof}
	We show that the sum $V^\one+V^\two$ satisfies Mantegazza's integration by parts formula: Let $\varphi\in C_c^1(\Omega\times G_{2,3})$. Then $A^{V^\one+V^\two}=A^{V^1}\chi_{\supp(V^\one)}+A^{V^2}\chi_{\supp(V^\two)}$ yields
	\begin{align*}
	&\int_{\Omega\times G_{2,3}}
	\left(P_{ij}\d_j\varphi+(\d_{P_{jk}}\varphi)\,A_{ijk}^{V^1+V^2}+\varphi\,A_{jij}^{V^1+V^2}\right)\dint(V^1+V^2)\\
	&=\int_{\Omega\times G_{2,3}}\left(P_{ij}\d_j\varphi+(\d_{P_{jk}}\varphi)\,A_{ijk}^{V^1}+\varphi\, A_{jij}^{V^1}\right)\dint V^1
	\\
&\quad +\int_{\Omega\times G_{2,3}}\left(P_{ij}\d_j\varphi+(\d_{P_{jk}}\varphi)\,A_{ijk}^{V^2}+\varphi\, A_{jij}^{V^2}\right)\dint V^2\\
	&=-\int_{\Omega\times G_{2,3}} \varphi\dint (\d V^1)_i-\int_{\Omega\times G_{2,3}} \varphi\dint (\d V^2)_i\\
	&=-\int_{\Omega\times G_{2,3}}\varphi\dint (\d V^1+\d V^2)_i=-\int_{\Omega\times G_{2,3}}  \varphi\dint (\d(V^1+V^2))_i,
	\end{align*}
	where the last equality follows from $\d V^1+\d V^2=\d(V^1+V^2)$.
\end{proof}

\begin{remark}[{{Sum of general curvature varifolds}}]\label{eichmann}
For $V^\one, V^\two\in AV_m(\Omega)$ with $\Omega\sbs\RR^n$ open we have \eqref{eq:sum_bdry} and $A^{V^\one+V^\two}=f^1A^{V^\one}+f^2A^{V^\two}$ where $f^i=\dint V^i/\dint(V^\one+V^\two)\in L^1_{\loc, V^\one+V^\two}(\Omega\times G_{m,n})$ for $i=\one,\two$.
The density (Radon-Nikodym derivative) $f^i$ always exists because of the absolute continuity $V^i\ll V^\one+V^\two$, as for all $U\subseteq \Omega\times G_{m,n}$ open,  $(V^\one+V^\two)(U)=0$ implies $V^i(U)=0$, for $i=\one,\two$. In case of mutual singularity, we recover the specific case of  \eqref{eq:sum_A}, for the densities reduce to the characteristic functions $f^i=\chi_{\supp(V^i)}$.
\end{remark}

We note that requiring a zero boundary measure of the sum $V^\one+V^\two$ rules out kinks at phase-boundaries. Indeed, consider two compact $C^2$ surfaces $M_\one$ and $M_\two$ in $\Omega$ with associated curvature varifolds 
$$V^i:=v(M_i,1)\in AV_2(\Omega)$$ 
for $i=1,2$. Their Mantegazza and Allard boundary measures are given by
$$
\d V^i=\delta_{\d M_i}\otimes\nu_i\,\delta_{q(\ast\nu_i)}
\qquad\text{and}\qquad
\pi_\sharp(\d V^i)=\nu_i\,\delta_{\d M_i}
$$
respectively. Here $\nu_i$ denotes the inwards-pointing unit normal of $\d M_i$ and $q(\ast\nu_i)\in G_{2,3}$ is the corresponding unoriented tangent plane.  Assume that the sets only overlap at their common boundary 
$$\d M_\one=\d M_\two=M_\one\cap M_\two\neq\emptyset$$ 
and satisfy the no-overlap condition \eqref{eq:nooverlap} (thanks to Lemma \ref{lem:nooverlap} this is the case if $M_\one$ and $M_\two$ have minimal distance $\veps_0$ elsewhere; cf.\ Figure \ref{fig:H1}). Then their sum reads
$$V^\one+V^\two=v(M^\one\cup M^\two,1)$$
and with \eqref{eq:sum_bdry} satisfies
$\d (V^\one+V^\two)=\d V^\one+\d V^\two
=\delta_{M_\one\cap M_\two}\otimes\left(\nu_\one\,\delta_{q(\ast\nu_\one)}+\nu_\two\,\delta_{q(\ast\nu_\two)}\right)$ 
and
$\pi_\sharp(\d (V^\one+V^\two))=(\nu_\one+\nu_\two)\,\delta_{M_\one\cap M_\two}$.
Consequently, if 
$$\d (V^\one+V^\two)=0,$$ 
then $\nu_\one=-\nu_\two$. 
Thus the unit normals must be antiparallel, showing that $M_\one\cup M_\two$ is $C^1$  and no kinks occur. 

We are in the position to present the main results of this section.

\begin{theorem}[Existence, two-phase case]
Fix $m^\one,m^\two,\veps_0,e >0$ and let the bounds \eqref{eq:bounds} hold true both for $\beta^\one,\gamma^\one$ and $\beta^\two,\gamma^\two$. Then there exists a pair $V^{\one},V^{\two}\in {\mathcal A}$ solving
\begin{align*}
\min\Big\{{\FCH}_{,2}(V^{\one},V^{\two}):
&\:\:V^{\one},V^{\two}\in\mathcal{A},\quad
\mu_{V^{\one}}(\Omega)=m^{\one},\quad \mu_{V^{\two}}(\Omega)=m^{\two},\quad \d(q_\sharp V^\one+q_\sharp V^\two)=0\\
& (\mu_{V^{\one}}\otimes\mu_{V^{\two}})(D_{\Omega,\veps}) \leq \veps^3/\veps_0\:\:\text{for all}\:\: \veps\in(0,\veps_0),\quad\encvol(V^\one+V^\two)=e\:\Big\}.
\end{align*}
\end{theorem}
\begin{proof}  We essentially follow the argument  for the single-phase case (Theorem \ref{thm:single}),  by letting  $V=V^{i}$ for $i=\one,\,\two$.  In order to reproduce estimate \eqref{eq:firstvarestimate}, we use now  
$$
\|\d (q_\sharp V^i)\|(\Omega\times G_{2,3})\leq c\, {\FCH}_{,2}(V^\one,V^\two).
$$
 Thus the  minimizing $V_h^i\in\cA$ ($i=\one,\,\two$) enjoy the first variation bounds
\begin{align*}
\|\delta(q_\sharp V_h^i)\|(\Omega)
&\leq \sqrt{2\,\mu_{V_h^i}(\Omega)}\, \|A^{q_\sharp V_h^i}\|_{L^2_{q_\sharp V_h^i}(\Omega\times G_{2,3})}
+\|\d (q_\sharp V_h^i)\|(\Omega\times G_{2,3})\nn\\
&\leq c'\left(\sqrt{\mu_{V_h^i}(\Omega)}\,\sqrt{{\FCH}_{,2}(V^\one,V^\two)+\mu_{V_h^i}(\Omega)}+{\FCH}_{,2}(V^\one,V^\two)\right)<\infty.
\end{align*} 
Lemma \ref{lem:nooverlap} gives closedness for the no-overlap condition. We can calculate the boundary measure of the sum of curvature varifolds via formula \eqref{eq:sum_bdry} and by Theorem \ref{thm:Mantegazza_Thm6} we obtain 
$$
\d (q_\sharp V_h^\one+q_\sharp V_h^\two)=\d (q_\sharp V_h^\one)+\d(q_\sharp V_h^\two)
\wto^\ast \d (q_\sharp V^\one)+\d(q_\sharp V^\two)=\d (q_\sharp V^\one+q_\sharp V^\two).
$$
This shows conservation of the zero boundary measure, because if $\d(q_\sharp V^\one_h+ q_\sharp V^\two_h)=0$, then also 
$\d(q_\sharp V^\one+ q_\sharp V^\two)=0$.  We note that the enclosed volume is defined for the sum $V^\one+ V^\two$ of oriented varifolds $\Radon(\Omega\times G_{2,3}^o)$. Therefore, the volume condition is stable under varifold limits also in the multiphase case.
\end{proof}

The existence result directly extends to sharp-interface $N$-phase membranes with energy
$$
\FCHN(V^1,\ldots,V^N)
:={\sum_{i=1}^N}\left(\FCH^{i}(V^i)+\sigma^i\,\|\d (q_\sharp V^i)\|\right)
$$
for $V^i\in\cA$ and $\sigma^i>0$ ($i=1,\ldots,N$).

\begin{theorem}[Existence, multiphase case]\label{thm:multi}
Let $m^i>0$ ($i=1,\ldots,N$) and $\veps_0, e >0$ be fixed. If $-{6} \,\beta^i<5\gamma^i<0$, for $i=1,\ldots,N$, then there exists a solution $(V^i)_{i=1}^N$ of
\begin{align*}
\min\Big\{&\FCHN(V^1,\ldots,V^N):\:\: V^i\in\mathcal{A},\quad \mu_{V^i}(\Omega)=m^i\:\: (i=1,\ldots,N),\quad \textstyle{\d(\sum_{i=1}^N q_\sharp V^i)}=0,\\
& \: \textstyle{\sum_{i,j=1,\:i\neq j}^N}\left(\mu_{V^i}\otimes\mu_{V^j}\right)(D_{\Omega,\veps}) \leq \veps^3/\veps_0\:\:\text{for all}\:\: \veps\in(0,\veps_0),\quad\textstyle{\encvol(\sum_{i=1}^N V^i)=e}\:\Big\}.
\end{align*}
\end{theorem}

As for the  single-phase case,  the area and enclosed-volume bounds $m^i$ and $e$ are required to respect the isoperimetric inequality
$$
(6\,\sqrt{\pi}\,e)^{1/3}\leq \left(\textstyle{\sum_{i=1}^N}m^i\right)^{1/2}
$$
to include smooth multiphase minimizers.

\subsection{Regularity and diameter bounds}

We collect here some remarks on all minimizers $V$  for  the single-phase problem (Theorem \ref{thm:single}) and  on all  individual phases $V^i$ of minimizers (or their sum)  for  the sharp-interface multiphase problem (Theorem \ref{thm:multi}). 

\begin{lemma}[$C^2$-rectifiability]\label{lem:menne}
Let $V\in AV_2^o(\Omega)$. Then, there exists a countable collection $(\Sigma_i)_{i=1}^\infty$ of $C^2$ surfaces in $\RR^3$ such that 
$$
\textstyle{\mu_V\left(\Omega\setminus\bigcup_{i=1}^\infty\Sigma_i\right)=0.}
$$
Moreover, for each $\Sigma_i$, the second fundamental forms of $V$ and $\Sigma_i$ $\mu_V$-a.e.\ agree on $\Omega\cap \Sigma_i$.
\end{lemma}

\begin{proof}
It suffices to observe $q_\sharp V\in AV_2(\Omega)$ and apply \cite[Theorem 1]{Menne:13}.
\end{proof}

In order to proceed further, we remark that all minimizers have finite Willmore energy \eqref{Willmore}, which in the context of varifolds (and up to normalization) is given by 
$$
\Will(W):=\int_{\Omega} |\bar H^W|^2\,\dint\mu_W=\int_{\Omega\times G_{2,3}} |\bar H^W(x,P)|^2\dint W(x,P).
$$

\begin{lemma}[Bound on the Willmore energy]\label{lem:Willmore}
Assume \eqref{eq:bounds}. Then, for all $V\in AV_2^o(\Omega)$ we have 
$$
\Will(V):=\Will(q_\sharp V)\leq 2\,c_1\left(\FCH(V)+c_2\,\mu_{V}(\Omega)\right),
$$
where $c_1$ and $c_2$ are the constants from Proposition
{\em\ref{prop:convexity}}. In particular, if $\mu_{V}(\Omega)<\infty$ and $\FCH(V)<\infty$, one deduces $\Will(V)<\infty$. 	
\end{lemma}

\begin{proof} Letting $W:=q_\sharp V\in AV_2(\Omega)$, the algebraic estimate $|\bar H^W|^2\leq 2|A^W|^2$, see 
\eqref{eq:HA}, gives
$$
\Will(W)\leq 2\int_{\Omega\times G_{2,3}} |A^W|^2\dint W=2\|A^W\|^2_{L^2_W(\Omega\times G_{2,3})}.
$$
Relation \eqref{eq:estimate} ensures that $\|A^{q_\sharp V}\|^2_{L^2_{q_\sharp V}(\Omega\times G_{2,3})}\leq
c_1\left(\FCH(V)+c_2\,\mu_{V}(\Omega)\right)$, so that the assertion follows. 
\end{proof}

Estimates on the Willmore energy are related to embeddedness. Indeed, in case of a smooth immersion $\psi\colon\Sigma\to\RR^3$ of a compact surface, if there exists $p\in\RR^3$ such that $\psi^{-1}(p)=\{x_1,\ldots, x_k\}$ (distinct points), then one has the Li-Yau inequality  \cite[Theorem 6]{LiYau:82}  
\begin{equation}\label{rem:LiYau}\widetilde{\Will}(\Sigma)\geq 4\pi k,
\end{equation} 
where $\widetilde{\Will}(\Sigma)=\frac{1}{4}\Will(\Sigma)$ is the Willmore energy \eqref{Willmore} normalized to $\widetilde{\Will}(S^2)=4\pi$. The consequences are as follows (see also \cite{KuwertSchaetzle:12}): 
\begin{enumerate}[(i)]
\item $\widetilde{\Will}(\Sigma)\geq 4\pi$ for embedded surfaces $\Sigma$; equality holds iff $\Sigma=S^{2}$.
\item $\widetilde{\Will}(\Sigma)\geq 8\pi$ for surfaces $\Sigma$ with self-intersections.
\item If $\widetilde{\Will}(\Sigma)< 8\pi$ then $\Sigma$ must be embedded.
\end{enumerate} 		

The next two lemmas exclude that the Cahnam-Helfrich varifold-minimizers collapse to a single point or expand to infinity by estimating their diameter. The statements are a varifold analogue of the diameter bounds given in \cite{Simon:93} for smooth connected hypersurfaces $\Sigma$ in $\RR^n$: If $\d\Sigma=\emptyset$ and $\Sigma$ is compact, then \cite[Lemma 1.1]{Simon:93} shows
$$
\sqrt{\frac{|\Sigma|}{\Will(\Sigma)}}\leq\diam(\Sigma)\leq c\sqrt{|\Sigma|\,\Will(\Sigma)}.
$$

For surfaces in $\RR^3$ it has been proven in \cite[Lemma 1]{Topping:98} that $c=2/\pi$. Moreover, if $\Sigma$ is a smooth connected hypersurface in $\RR^n$ whose boundary $\d\Sigma\neq\emptyset$ has finitely many connected components $\Gamma_i$, then one has the upper bound \cite[Lemma 1.2]{Simon:93}
$$
\diam(\Sigma)\leq C\left(\int_{\Sigma}(\kappa_1^2+\kappa_2^2)^{1/2}+\sum_{i}\diam(\Gamma_i)\right).
$$

\begin{lemma}[Lower diameter bound]\label{lem:lowerdiam}
Assume \eqref{eq:bounds} and let $V\in AV_2^o(\Omega)$ with $\mu_V(\Omega)=m>0$ satisfy $\|\d(q_\sharp V)\|(\Omega\times
G_{2,3})<\infty$ and $\FCH(V)<\infty$. Then, we have that $\diam(\supp\,\mu_V)>0$.
\end{lemma}

\begin{proof}
We argue along the lines of \cite[Theorem 3.1.1]{Wojtowytsch:17thesis} but generalize the proof to varifolds with nonzero Mantegazza boundary
measure. With no loss of generality, let $0\in\supp\,\mu_V$ and consider $X\in C_c^1(\Omega;\RR^3)$ with 
$$
X(x)=x\quad\text{on}\quad \supp\,\mu_V \quad\text{and} \quad\|X\|_{L^\infty(\Omega)}=\sup_{x\in\Omega}|X(x)|\leq c\,\diam(\supp\,\mu_V)
$$
for some $c>1$. Then, for all (projections on) planes $P\in G_{2,3}$ and $x\in\supp\,\mu_V$,
$$
\div^PX(x)=P_{ij}\d_jX_i(x)=2.
$$ 
Consequently, Mantegazza's first variation formula \eqref{eq:Mantegazzaformula} gives (here we write $V$ instead of $q_\sharp V=v(M,\theta_++\theta_-)$ to shorten the expressions)
\begin{align*}
2\,\mu_V(\Omega)&=\int_{\Omega}\div^{P(x)}X(x)\dint\mu_V(x)\\
&=\int_{\Omega}\lara{X(x),\bar H^V(x,P(x))}\dint\mu_V(x)
-\int_{{\Omega}\times G_{2,3}}\lara{X(x),\dint(\d V)(x,P)}\\
&\leq\left(\int_{\Omega}|X|^2\dint\mu_V\right)^{1/2}\left(\int_{\Omega}|\bar H^V(\cdot,P)|^2\dint\mu_V\right)^{1/2}
+\int_{{\Omega}\times G_{2,3}}|\lara{X,\dint(\d V)}|\\
&\leq\|X\|_{L^\infty({\Omega})}\left(\int_{\Omega}\dint\mu_V\right)^{1/2}\left(\int_{{\Omega}\times G_{2,3}}|\bar H^V|^2\dint V\right)^{1/2}
+\|X\|_{L^\infty({\Omega})}\int_{{\Omega}\times G_{2,3}}\dint\|\d V\|\\
&\leq c\,\diam(\supp\,\mu_V)\left(\left(\mu_V({\Omega})\int_{{\Omega}\times G_{2,3}}|\bar H^V|^2\dint V\right)^{1/2}
+\|\d V\|(\Omega\times G_{2,3})\right).
\end{align*}
Hence
$$
2\mu_V(\Omega)\leq c\,\diam(\supp\,\mu_V)\left(\sqrt{\mu_V(\Omega)\,\Will(V)}
+\|\d V\|(\Omega\times G_{2,3})\right).
$$
As $\mu_V(\Omega)=m>0$, $\|\d V\|(\Omega\times G_{2,3})<\infty$, and $\Will(V)<\infty$ (see Lemma \ref{lem:Willmore}) one has the lower bound
$$
0<\frac{\mu_V(\Omega)}{\sqrt{\mu_V(\Omega)\,\Will(V)}+\|\d V\|(\Omega\times G_{2,3})}\leq\diam(\supp\,\mu_V)
$$
and the statement follows. 
\end{proof}

Note that the statement of Lemma \ref{lem:lowerdiam} remains true if the condition $\mu_V(\Omega)=m$ is replaced by  $\mu_{V}(\Omega)>0$. 

\begin{lemma}[Upper diameter bound]\label{lem:lowerdiam2}
Let $\Omega=\RR^3$ and, in addition to the assumptions of Lemma \emph{\ref{lem:lowerdiam}}, assume $\d(q_\sharp V)=0$. Then  $\diam(\supp\,\mu_V)<\infty$.
\end{lemma}

\begin{proof}
The upper diameter bound follows from \cite[Lemma 1.1]{Simon:93} or \cite[Lemma 1]{Topping:98} for smooth surfaces, which can be generalized to the varifold setting \cite[Theorem 3.1.1]{Wojtowytsch:17thesis}.
\end{proof}

The latter result is based on a refined version of Li-Yau inequality \eqref{rem:LiYau}, which seems not to hold for varifolds with boundary. Consequently, as it stands, the upper bound only directly applies to minimizers $V$ of the single-phase problem, but not to the individual phases $V^i$ in the multiphase case. However, by formula \eqref{eq:sum_bdry} and Theorem \ref{thm:multi} the sum $\sum_{i=1}^N q_\sharp V^i$ (corresponding to the union of sets) is a varifold without boundary, for which the upper diameter bound holds. A fortiori, also the individual phases have bounded diameter. 

The upper diameter bound eventually allows us to remove the constraint $\supp\,\mu_V\sbs K$ for a given compact $K\sbs\Omega$ from the definition of admissible varifolds $V\in\cA$ for Canham-Helfrich minimizers,  see \eqref{eq:admissible}. 	

\section*{Acknowledgments}
The authors gratefully acknowledge some valuable comments by Sascha Eichmann on a former version of the paper.
Luca Lussardi is grateful to the kind hospitality of the University of Vienna where part of the work was performed. This work has been partially supported by the Vienna Science and Technology Fund (WWTF) through the project MA14-009 and by the Austrian Science Fund (FWF) projects F\,65 and W\,1245.

\newcommand{\SortNoop}[1]{}

\bibliographystyle{plain}
\end{document}